\documentclass[12pt]{amsart}
\usepackage{psfrag,epsfig,amssymb}
\numberwithin{figure}{section}
\numberwithin{table}{section}

\newtheorem{theorem}{Theorem}[section]

\newtheorem{lemma}[theorem]{Lemma}
\newtheorem{prop}[theorem]{Proposition}

\theoremstyle{definition}
\newtheorem{definition}[theorem]{Definition}

\newtheorem{cor}[theorem]{Corollary}

\newtheorem{notation}[theorem]{Notation}
\theoremstyle{remark}
\newtheorem{remark}[theorem]{Remark}
\numberwithin{equation}{section}

\def \H{{\mathbb H}}

\def \R{{\mathbb R}}

\def \h{{\mathfrak h}}
\def \Z{{\mathbb Z}}
\def \SS{{\mathbb S}}

\def \[{[ }
\def \]{] }
\def \C{{\mathcal C}}
\def \M{{\mathcal M}}
\def \S{{\mathcal S}}

\def \L{{\mathcal L}}

\def \tr{\text{\, tr}}

\def \C{\mathfrak C}
\def \T{\mathcal T}
\def \M{\mathcal M}

\begin{document}

\author{Anna Felikson}
%\address{Independent University of Moscow, B. Vlassievskii 11, 119002 Moscow, Russia}
%\curraddr{Department of Mathematics, University of Fribourg, P\'erolles, Chemin du Mus\'ee 23, CH-1700 Fribourg, Switzerland}
%\email{felikson@mccme.ru}
%\address{Max-Planck-Institut fur Mathematik,Vivatsgasse 7, 53111 Bonn, Germany}
%\email{felikson@mpim-bonn.mpg.de}
\address{Independent University of Moscow}
\curraddr{Jacobs University Bremen}
\email{a.felikson@jacobs-university.de}
\thanks{Research of the second author is supported by grants RFBR 10-01-00678, NSh 8462.2010.1.}

\author{Sergey Natanzon}
\address{Higher School of Economics, Moscow  \phantom{wwwwwwwwwwwwwwwwwwwwwwwwwwww}
%\phantom{B}
Belozersky Institute of Physico-Chemical Biology, Moscow State University,\phantom{wwwww}
Institute Theoretical and Experimental Physics}
\email{natanzons@mail.ru}

%

%\author{Michael Shapiro}
%\address{Department of Mathematics, Michigan State University, East Lansing, MI 48824, USA}
%\email{mshapiro@math.msu.edu}

%    author two information
%\author{Pavel Tumarkin}
%\address{Independent University of Moscow, B. Vlassievskii 11, 119002 Moscow, Russia}
%\curraddr{Department of Mathematics, University of Fribourg, P\'erolles, Chemin du Mus\'ee 23, CH-1700 Fribourg, Switzerland}
%\email{pasha@mccme.ru}
%\date{November11, 2008}

%\title{On parameterizations of moduli space by lengths of closed geodesics}
\title[Moduli via double pants decompositions]{Moduli via double pants decompositions}

\begin{abstract} 
We consider (local) parametrizations of Teichm\"uller space $\T_{g,n}$ (of genus $g$ hyperbolic surfaces with $n$ boundary components) 
by lengths of $6g-6+3n$ geodesics. We find a large family of suitable sets of  $6g-6+3n$ geodesics, each set forming a special 
structure called ``admissible double pants decomposition''. For admissible double pants decompositions containing no double curves
 we show that the lengths of 
curves contained in the decomposition determine the point of $\T_{g,n}$ up to finitely many choices. Moreover, these lengths provide
a local coordinate in a neighborhood of all points of $\T_{g,n}\setminus X$ where $X$ is a union of $3g-3+n$ hypersurfaces. 
Furthermore, there exists a groupoid acting transitively on admissible double pants decompositions and generated by  
transformations exchanging only one curve of the decomposition. The local charts arising from different double pants decompositions 
compose an atlas  covering the Teichm\"uller space. The gluings of the adjacent charts are coming from the elementary transformations
 of the decompositions, the gluing functions are algebraic.
The same charts provide an atlas for a large part of the boundary strata in Deligne-Mumford compactification of the moduli space
$\M_{g,n}$.
%
%atlas works also for $\M_{g,n}\cup \S_2$ where $\M_{g,n}$ is a moduli space and $\S_2$ is a union of codimension 2 strata of
%Deligne-Mumford compactification.

\end{abstract}

\maketitle

\setcounter{tocdepth}{1}
\tableofcontents

\section*{Introduction}

Consider a hyperbolic structure on a closed oriented surface $S_{g,n}$, $2g+n>2$, of genus $g$ with $n$ boundary components. 
%The Teichm\"uller space $\T=\T_{g,n}$ of the marked hyperbolic structures is the set of pairs $(h,\phi)$,
% where $h$ is a hyperbolic structure on $S_{g,n}$ and $\phi$ is the homeomorphism $\Phi: S\to S'$. ???  
In~\cite{FK}, Fricke and Klein proved that in case $n=0$ the Teichm\"uller space   $\T=\T_{g,n}$ for such a surface
 is homeomorphic to $(6g-6)$-dimensional Euclidean space. Moreover, they specified a point of Teichm\"uller space
by the lengths of closed geodesics contained in some (rather large) set. 
%This set was later used by Wolpert~\cite{W} to show that $6g-6$ length parameters are sufficient to serve as local parameters in Teichm\"uller space.  

After Fricke and Klein many authors investigated various sets of global parameters on the Teichm\"uller space.
Fenchel and Nielsen~\cite{Fen} introduced ``length-twists'' coordinates which in case of closed surface 
consist of $3g-3$ lengths of mutually 
non-intersecting geodesics and $3g-3$ twist parameters along them. Natanzon~\cite{N} described a convenient set of parameters (including 
both lengths of geodesics and parameters of other nature), 
%which provided a shorter proof of the fact that the Teichm\"uller space is homeomorphic to a cell. 
allowing to recover the Fuchsian group of the surface.
A lot of efforts were spent on descriptions of purely length global parameters, especially, 
for the question of minimal possible number of geodesics whose lengths are sufficient to serve as a global coordinate
on the Teichm\"uller space. First, it was  shown that $9g-9$ length of geodesics may serve as global parameters in $\T_{g,0}$. 
Later,  Wolpert~\cite{W1} used the construction of Fricke and Klein to show that $6g-6$ lengths are sufficient for a local coordinate
in $\T_{g,0}$ (but not for a global one).
It was natural to expect that $6g-6$ lengths of geodesics can serve as a global coordinate on $\T_{g,0}$, however,
Wolpert~\cite{W2} showed that $\T_{g,0}$ can not be parametrized globally by lengths of $6g-6$ geodesics.
%
%it is known for a long time that $6g-6$ lengths of geodesics are insufficient for this aim (the fact is recited in numerous papers,
%but we have failed to find any reference to a proof in the literature).   
Sepp\"al\"a and Sorvali~\cite{SS3} presented a global parameterization of $\T_{g,0}$ 
by $6g-4$ length functions (as a by-product they also gave an example of $6g-6$ length parameters defining the surface up to at most 
4 possibilities). 
Finally, in~\cite{S}  Schmutz obtained a global parameterization by $6g-5$ lengths of  geodesics, which is due to~\cite{W2} is 
minimal possible. Another example of such a minimal parameterization is given in~\cite{H1} by Hamenst\"adt. 
In the case of surfaces with cusps or holes the situation is easier: the $(6g-6+2m+3n)$-dimensional Teichm\"uller space
of surfaces with $m$ cusps and $n$ holes may be globally parametrized  by  $(6g-6+2m+3n)$ length parameters 
(see ~\cite{SS3},~\cite{S} and~\cite{H1}).
Hamenst\"adt~\cite{H2} also showed that such a parametrization may be extended to the Thurston boundary of $\T$.

\medskip

In this paper, we consider the Teichm\"uller space $\T=\T_{g,n}$ of marked hyperbolic structures on an oriented surface $S=S_{g,n}$,
$2g+n>2$ of genus $g$ with $n$ geodesic boundary components. The dimension of this space is $6g-6+3n$, so we are interested in
sets of $6g-6+3n$ curves on $S$ whose lengths parametrize $\T$. We build a large family of the sets of $6g-6+3n$ curves such that
the lengths of curves from each set determine a point of $\T$ up to finitely many possibilities and provide a local coordinate
in neighborhoods of most points of $\T$, the local charts of this type compose an atlas on $\T$, 
the transition functions between the charts are algebraic.
Moreover, the same atlas works  for regular points of the moduli the space 
$\M=\T/Mod$ (where $Mod$ is a modular group) and covers also a large part 
of the Deligne-Mumford compactification of $\M$.

In more details, we build a large family of the sets of $6g-6+3n$ curves
on $S$ satisfying the following properties:
\begin{itemize}
\item[{\bf 1.}] (Parametrizing property). The lengths of the curves of each set determine a point of $\T$ up to finitely many choices;
they provide a local coordinate in the neighborhoods of almost all points of $\T$.

\item[{\bf 2.}] (Double pants decomposition property). Each set compose an {\it admissible double pants decomposition} defined and
studied recently in~\cite{FN}; it consists of two pants decompositions (where a {\it pants decomposition} is a set of curves decomposing
the surface into ``pairs of pants'', i.e. into spheres with 3 holes). Each pants decomposition defines a handlebody with $S$ as the 
boundary, so, two pants decompositions define a Heegaard splitting of some 3-manifold $M^3$. 
The {\it admissible} double pants decompositions are ones corresponding to Heegaard splittings of the 3-sphere
(there exists also an equivalent combinatorial definition which is used throughout the proofs).

\item[{\bf 3.}] (Groupoid action). There exists a groupoid acting on admissible double pants decompositions transitively and generated
by simple transformations of two types (called ``flips'' and ``handle-twists''), each of the generating transformations changes 
exactly one curve of a double pants decomposition. The length of the new curve is an algebraic function of the lengths of the initial 
curves.

\item[{\bf 4.}] (Atlas on $\T$ with algebraic transition functions). The charts arising from admissible double pants decompositions  
compose an atlas on $\T$; the transition functions between the charts are algebraic.

\item[{\bf 5.}] (Extension to most strata of Deligne-Mumford compactification). Let $Mod$ be a modular group of $S$ and let $\M=\T/Mod$
be the corresponding moduli space.
Each point of the Deligne-Mumford compactification $\overline \M$ of $\M$ is a boundary point for some chart coming from a double 
pants decomposition.
Moreover, for most points of $\overline \M$  
(including almost all points of the strata of minimal codimension)
there exists a chart coming from a double pants decomposition and covering a neighborhood of the point in the corresponding stratum
as well as covering almost all point in the neighborhood of the point in $\overline \M$.

\end{itemize}

More precisely, let $DP$ be an admissible double pants decomposition whose curves are closed geodesics in $S$. 
In principle, two pants decompositions contained in $DP$ may have a common curve (called a {\it double curve}),
we will be interested in double pants decompositions containing no double curves.
%Suppose that all curves of $DP$ are distinct, i.e. $DP$ contains no double curves.
Let $l(DP)$ be the ordered set of lengths of curves composing $DP$. Then we prove the following:

\medskip
\noindent
{\bf Theorem A.} (see Theorem~\ref{local} below).
Let $DP$ be an admissible double pants decomposition without double curves.
Then $DP$ together with the ordered set of lengths $l(DP)=\{l(c_i)| c_i\in DP \}$ is a local coordinate in $\T\setminus Z$
where $Z$ is a union of finitely many codimension 1 subsurfaces in $\T$ (each homeomorphic to a codimension 1 disk).

\medskip

Moreover, we also prove the following result.

\medskip
\noindent
{\bf Theorem B.} (see Theorem~\ref{finite} below).
Let $DP$ be an admissible double pants decomposition containing no double curves. 
Then $l(DP)$ determines a point of $\T$ up to finitely many choices.

\medskip

Composing Theorems~A and~B with the fact (see ~\cite{FN}) that there exists a groupoid acting on admissible double pants decompositions
transitively, we derive the following theorem.

\medskip
\noindent
{\bf Theorem C.} (see Theorem~\ref{trans on charts} below).
(1) The charts with coordinates $l(DP)$, where $DP$ is an admissible double pants decomposition without double curves,
provide an atlas on Teichm\"uller space $\T$.

(2) The elementary transition functions of these charts are induced by  elementary transformations of double pants decompositions,
each elementary transition function change only one coordinate. 
This unique non-trivial transition function is algebraic.

(3) The compositions of elementary transition functions act transitively on the charts.

\medskip

The structure of double pants decomposition is convenient to work with Deligne-Mumford compactification of the moduli space. 
Let $C$ be a set of mutually disjoint simple curves on $S$.
Contracting the curves contained in $C$ we obtain a point of the compactification, on the other hand, we stay in any chart arising
from a double pants decomposition $DP$ such that $C\in DP$ (more precisely, the limit point belongs to the boundary of the chart),
see Theorem~\ref{closure} and Corollary~\ref{cor}.

Furthermore, contraction of the curves of $C$ turns a conveniently chosen double pants decomposition $DP$
into a double pants decomposition of the obtained surface with nodal singularities (provided that $C\in DP$ and each curve of $C$
is intersected by a unique other curve of $DP$). There are some cases when such a convenient decomposition does not exist,
however, for the most configuration of curves $C$ we show that it does exist. In this case we say that the set $C$ is {\it good}
and the stratum $\S_C\in \overline \M$ is {\it good} (here $S_C$ is the set of nodal surfaces obtained by shrinking all curves of $C$,
$\M$ is the moduli space and $\overline \M$ is its Deligne-Mumford compactification).
In particular, all strata of minimal codimension (i.e. of codimension 2) are good strata.
For a good set of curves $C$ we define another length-type coordinates as  
$ \tilde l(DP,C)=\{l(c_i), \frac{1}{l(c_j)} \ | \ c_i\in C, c_j\in  DP\setminus C  \}$. 
We show that the functions $\tilde l(DP,C)$ produce {\it almost charts} covering the good strata of $\overline \M$,
i.e. given a point $\tau'\in \S_C$ in a good stratum $\S_C$ there exists an admissible double pants decomposition $DP$ 
and a neighborhood $O(\tau')\subset \overline \M$ in a natural topology such that
  $\tilde l(DP,C)$ produce a local coordinate in $O(\tau')\cap \S_C$ and give a local coordinate
in some set $O(\tau')\setminus Z \in \overline \M$, where $Z$ is a union of finitely many codimension 1 subsurfaces in $\M$.
More precisely, we prove the following theorem.

\medskip
\noindent
{\bf Theorem D.} (see Theorem~\ref{atlas on M} below).
Let $S$ be a nodal surface, let $\M(S)$ be its moduli space and let $\overline \M(S)$ be Deligne-Mumford compactification of $\M$.
Let $\S_{good}^{\M}=\S_{good}/Mod$ be the union of good strata in $\M$.    
Let $O$ be a locus of orbifold points of $\M$, let $\overline O$ be the closure of $O$ in $\overline \M$.
Then 

\begin{itemize}
\item[(1)] the charts with coordinates $\tilde l(DP,C)$  provide an atlas on $\M\setminus O$ and on $\S_{good}^\M\setminus \overline O$, 
(here  $C$ is a good set and $DP$ is an admissible double pants decomposition without double curves);

\item[(2)] each point $\tau'\in S_{good}^\M\setminus \overline O$ is covered by some almost chart  $(O'(\tau'),\tilde l(DP,C))$;

\item[(3)] the elementary transition functions of these charts (almost charts) 
%are inversions and transformations induced by  flips, quasi-handle-twists of double pants decompositions;
%each elementary transition function change only one coordinate; 
change only one coordinate,
this unique non-trivial transition function is algebraic;

\item[(4)] the compositions of elementary transition functions act transitively on the union of charts and almost charts.

\end{itemize}

\medskip

The paper is organized as follows. In Section~\ref{dp}, we recall from~\cite{FN} the definition of double pants decompositions 
and their properties. In Sections~\ref{sec-coord} and~\ref{sec tichnic}, we discuss Fenchel-Nielsen coordinates on $\T$,
and use them to prove some technical lemmas.
In Section~\ref{sec local}, we prove Theorem~A, i.e. we prove that double pants decompositions induce some local charts on $\T$ 
(see Theorem~\ref{local}). Section~\ref{sec finite} is devoted to the proof of Theorem~B (see Theorem~\ref{finite}).
In Section~\ref{atlas}, we collect the above mentioned local charts into an atlas on $\T$, this leads to Theorem~C  
(see Theorem~\ref{trans on charts}). Finally, in Section~\ref{compactification} we consider Deligne-Mumford compactification
 of the moduli space and prove Theorem~D  (see Theorem~\ref{atlas on M}).

\medskip
\noindent
{\bf Acknowledgments.} We are grateful to Antonio Costa, Dan Margalit and Saul Schleimer for helpful discussions. 
The work was partially written during the first author's stay at Max Planck Institute for Mathematics in Bonn. 
We are grateful to the Institute for hospitality, support and a nice working atmosphere.

\section{Preliminaries~I: double pants decompositions}
\label{dp}
In this section we recall from~\cite{FN} the definition of double pants decompositions and their properties.

\subsection{Pants decompositions}

Let $S=S_{g,n}$ be an oriented closed surface of genus $g\ge 0$ with $n$ boundary components.
We assume $2g+n> 2$,  which excludes spheres with less than 3 holes and the torus.
The surface $S$ is fixed throughout the paper.

A {\it curve} $c$ on $S$ is an embedded closed non-contractible non-selfintersecting curve considered up to a homotopy of $S$. 

Given a set of curves we always assume that there are no ``unnecessary intersections'', so that if two curves of this set
intersect each other in $k$ points then there are no homotopy equivalent pair of curves intersecting in less than $k$ points.

For a pair of curves $c_1$ and $c_2$ we denote by $|c_1\cap c_2|$ the number of (geometric) intersections of $c_1$ with $c_2$.

\begin{definition}[{{\it Pants decomposition}}]
A {\it pants decomposition} of $S$ is a set of (non-oriented) mutually disjoint curves 
$P=\{ c_1,\dots,c_{k} \}$  decomposing  $S$ into pairs of pants 
(i.e. into spheres with 3 holes).
In this paper, all boundary curves of $S$ are considered as a part of each pants decomposition of $S$.
\\
\end{definition}

It is easy to see that any pants decomposition of $S_{g,n}$ consists of $3g-3+2n$
 (where $3g-3+n$ curves decompose $S$ and $n$ curves are boundary curves).
%, together with the boundary curves this compose $m=3g-3+n$ curves.
Note, that we do allow self-folded pants, two of whose boundary components are identified in $S$.  
A surface which consists of one self-folded pair of pants will be called {\it handle}.

A curve $c\in P$,  is {\it regular} if $c\notin \partial S$ and $c$ is not a self-identified boundary curve of the self-folded 
pair of pants (i.e. if it is not lying inside a handle cut out by a curve $c'\in P$).

%An obvious computation shows that any flip preserves the Lagrangian plane defined by the pants decomposition. 
%i.e. $\L(f_i(P))=\L(P)$.

\begin{definition}[{{\it Flip}}]
Let $P=\{ c_1,\dots,c_{3g-3+2n}  \}$ be a pants decomposition.
Define a {\it flip  of $P$ in a regular curve} $c_i$  as  
a replacing of $c_i\subset P$ by any curve $c_i'$ satisfying the following properties:
\begin{itemize}
\item $c_i'$ does not coincide with any of $c_1,\dots,c_{3g-3+2n}$;
\item $|c_i'\cap c_i|=2$;
\item $c_i'\cap c_j=\emptyset$ for all $j\ne i$.

\end{itemize}

\end{definition}

See Fig.~\ref{fig-flip} for an example of a flip. Clearly, an inverse operation to a flip is also a flip (so that the set of flips
compose a groupoid acting on pants decompositions).

\begin{figure}[!h]
\begin{center}
\psfrag{c1}{\scriptsize $c'$}
\psfrag{c}{\scriptsize $c$}
\epsfig{file=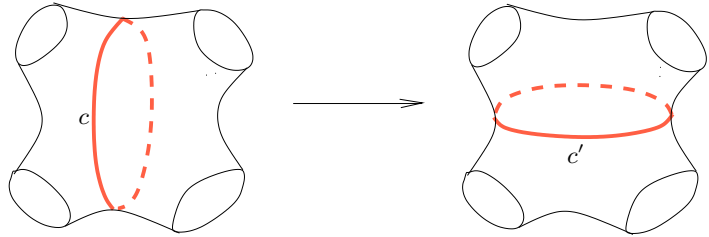,width=0.6\linewidth}
\caption{Flips of pants decomposition.} 
\label{fig-flip}
\end{center}
\end{figure}

\begin{definition}[{{\it Standard decomposition}}]
A decomposition $P$ of $S_{g,n}$ is {\it standard} if $P$ contains $g$ curves $c_1,\dots,c_g$ such that $c_i$, $i=1,\dots,n$, 
cuts out a handle. 

\end{definition}

\subsection{Double pants decompositions }

%\begin{definition}[{{\it Lagrangian plane of pants decomposition}}]
Let $P=\{ c_1,\dots,c_{3g-3+2n}\}$ be a pants decomposition. A {\it Lagrangian plane}  
$\L(P)\subset H_1(S,\Z)$ is a subspace spanned by the homology classes $h(c_i)$, $i=1,\dots,3g-3+2n$ 
(here $c_i$ is taken with any orientation).

%\end{definition}

%\begin{definition}[{{\it Lagrangian planes in general position}}]
Two  Lagrangian planes $\L(P_1)$ and $\L(P_2)$  are {\it in general position} if 
 $\L_1\cap \L_2=0$ and $H_1(S,\Z)=\langle \L_1,\L_2\rangle$ (where $\langle \L_1,\L_2\rangle$ denotes the sublattice of
$H_1(S,\Z)$ spanned by $\L_1$ and $\L_2$).

%\end{definition}

\begin{definition}[{{\it Double pants decomposition}}]
A {\it double pants decomposition} $DP=(P_a,P_b)$ is a pair of pants decompositions $P_a$ and $P_b$ of the same surface
such that the Lagrangian planes $\L_a=\L(P_a)$ and  $\L_b=\L(P_b)$ spanned by these pants decompositions
are in general position. $P_a$ and $P_b$ are called {\it parts} of $DP$.

\end{definition}

See Fig.~\ref{double_pant} for an example of a double pants decomposition.

\begin{figure}[!h]
\begin{center}
\psfrag{P}{$P_a$}
\psfrag{Z}{$P_b$}
\epsfig{file=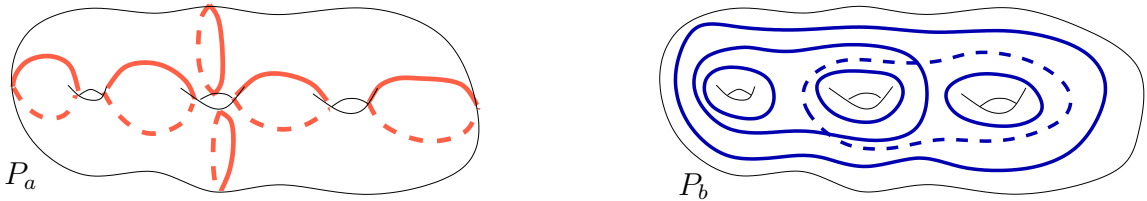,width=0.98\linewidth}
\caption{A double pants decomposition $(P_a,P_b)$. } 
\label{double_pant}
\end{center}
\end{figure}

There are several natural transformations on the set of double pants decompositions:
\begin{itemize}
\item flips of $P_a$;
\item flips of $P_b$;
\item handle-twists (see Definition~\ref{Dehn} below). 

\end{itemize}

\begin{definition}[{{\it Handle-twists}}]
\label{Dehn}
Given a double pants decomposition $DP=(P_a,P_b)$
we define an additional transformation which may be performed if both parts $P_a$ and $P_b$ contain the same  curve $a_i=b_i$  
separating the same handle $\h$, 
see Fig.~\ref{d-self-pant}(a). Let $a\in \h$ and $b\in \h$ be the only curves from $P_a$ and $P_b$ respectively.
Then a {\it handle-twist} $t_a(b)$ (respectively, $t_b(a)$) is a Dehn twist along $a$  (respectively, $b$) in any of two directions
(see Fig.~\ref{d-self-pant}(b)).  

\end{definition}

\begin{figure}[!h]
\begin{center}
\psfrag{a}{\scriptsize $a$}
\psfrag{b}{\scriptsize $b$}
\psfrag{a1}{\scriptsize $a'$}
\psfrag{b1}{\scriptsize $b'$}
\psfrag{a2}{\scriptsize $a'$}
\psfrag{b2}{\scriptsize $b'$}
\psfrag{aa}{\small (a)}
\psfrag{bb}{\small (b)}
\psfrag{cc}{\small (c)}
\psfrag{i}{\scriptsize $a_i=b_i$}
\epsfig{file=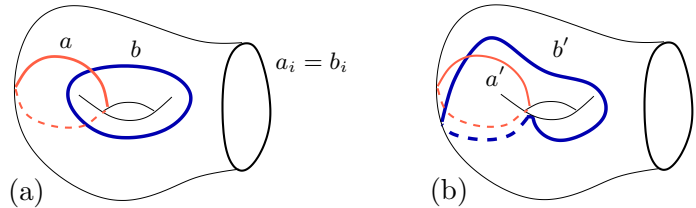,width=0.59\linewidth}
\caption{Handle-twists:
(a) Double self-folded pair of pants; (b) The same pair of pants after a handle-twist $t_a(b)$} 
\label{d-self-pant}
\end{center}
\end{figure}

Notice that both flips and handle-twists are reversible transformations, so that flips and handle-twists generate a groupoid acting on
the set of double pants decompositions.

%\begin{definition}[{{\it Flip-twist groupoid}}]
%A {\it flip-twist groupoid} $FT$ is a groupoid generated by flips and twists. 
%%(and containing identity morphisms for each double pants 
%%ddecompositions).
%
%\end{definition}

\begin{definition}[{{\it Double curve}}]
A curve $c\in (P_a,P_b)$ is {\it double} if $c\in (P_a\cap P_b)$ and $c\notin \partial S$. 

\end{definition}

\begin{definition}[{{\it Standard decomposition}}]
A double pants decomposition $(P_a,P_b)$ of $S_{g,n}$ is {\it standard} if 
there exist $g$ double curves $c_1,\dots,c_g \in (P_a,P_b)$ such that
 $c_i$ cuts out of $S$ a handle $\h_i$. 

A standard double pants decomposition $(P_a,P_b)$ is {\it strictly standard} if $(P_a,P_b)$ contains $2g-3+n$ double curves
(i.e. $c\in \{P_a\cup P_b\}\setminus \{P_a\cap P_b\}$ if and only if $c$ is contained inside some handle).

\end{definition}

See Fig.~\ref{standard double_pant} for an example of a standard double pants decomposition
(this decomposition may be turned into a strictly standard one in one flip).

\begin{figure}[!h]
\begin{center}
\psfrag{a}{$P_a$}
\psfrag{b}{$P_b$}
\epsfig{file=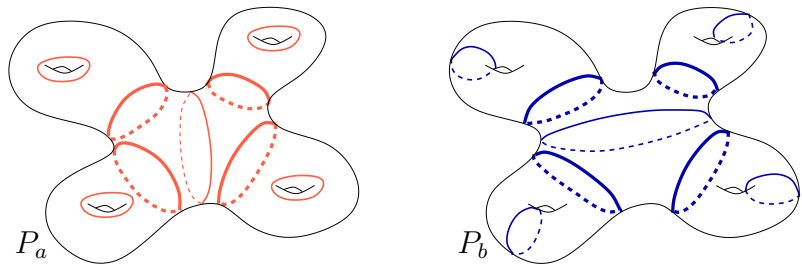,width=0.68\linewidth}
\caption{A standard double pants decomposition $(P_a,P_b)$.} 
\label{standard double_pant}
\end{center}
\end{figure}

\begin{definition}[{{\it Admissible decomposition}}]
A double pants decomposition $(P_a,P_b)$ is {\it admissible}  if it is possible to transform $(P_a,P_b)$ to a standard pants 
decomposition by a sequence of flips.

\end{definition} 

For example, the decomposition shown in Fig.~\ref{double_pant} is admissible.

The following theorem is the main result of~\cite{FN}.

%\begin{theorem}[\cite[Main Theorem]{FN}]
\begin{theorem}[\cite{FN}]
\label{trans} 
A groupoid generated by flips and handle-twists acts transitively on admissible double pants decompositions of $S=S_{g,n}$ 
(for any $(g,n)$ such that $2g+n>2$).

\end{theorem}

\begin{remark}[Admissible double pants decompositions and Heegaard splitting of $\SS^3$]
A set of admissible double pants decompositions have an invariant topological description in terms of 
Heegaard splittings of 3-manifolds.
For each pants decomposition $P$ of $S$ one may construct a handlebody $S_+$ such that $S$ is 
the boundary of $S_+$  and all curves of $P$ are contractible inside $S_+$. A union of two pants 
decompositions of the same surface define two different handlebodies bounded by $S$. Attaching this
handlebodies along $S$  one obtains a Heegaard splitting of some 3-manifold $M^3(DP)$.
It is shown in~\cite{FN} that  a pants decomposition $DP$ is admissible if and only if $M^3(DP)=\SS^3$,
where $\SS^3$ is a 3-sphere.

\end{remark}

We will also use the following result proved in~\cite[Lemma~6.1]{FN}.

\begin{prop}[\cite{FN}]
\label{ideal}
Let $S=S_{g,n}$, $2g+n>2$, and $Mod(S)$ be its modular group.   
Let $(P_a,P_b)$ be an admissible double pants decomposition without double curves.
Then  $\gamma\in Mod(S)$ fixes $(P_a,P_b)$ if and only if $\gamma=id$. 

\end{prop}

\section{Preliminaries~II: coordinates on Teichm\"uller space}
\label{sec-coord}

Let $S=S_{g,n}$ be a hyperbolic surface of genus $g$ with $n$ boundary components. Each boundary component is assumed to be a 
geodesic of finite length.
%, the punctures will be considered as boundary components with boundary geodesic of length 0.
%Denote $m=3g-3+n$ (this is the number of curves in a pants decomposition of $S$).

A Teichm\"uller space $\T=\T_{g,n}$ is a parameter space of marked hyperbolic metrics on the surface $S_{g,n}$.
For the marking on $S$ we will usually use  admissible double pants decompositions containing no double curves
(this provides a correct marking since any elements $\gamma\ne e$ of the modular group 
$Mod(S_{g,n})$ acts non-trivially on the decomposition, see~\cite[Lemma~6.1]{FN}). 
%Abusing the notation we denote $\T_{g,n}$ by $\T$.   

We will use Fenchel-Nielsen parameterization of the Teichm\"uller space. 
We shortly explain the parametrization below and refer to~\cite{T} for the details. 

To build the parameterization one chooses a pants decomposition
$P$ of $S$. Each pair of pants is uniquely determined by the lengths of its boundary curves. To encode the concrete hyperbolic structure
one need also to now how the adjacent pairs of pants a sewed together: one can choose an arbitrary way to attach them, and then rotate 
one piece along another by any real angle. More precisely, to determine the angle of the rotation one does the following:
\begin{itemize}
\item[1)] for each pair of pants 
$p^k\in P$ one chooses three disjoint segments $s_{ij}^k$, $i,j\in \{1,2,3\}$ orthogonal to the boundary components $b_{i}^k$ and 
$b_j^k$ of $p^k$ (so that $p^k$ is decomposed into two  right-angled hexagons); 
\item[2)] then one fixes some way to attach the adjacent pairs of pants $p^k$ and $p^{k'}$ so that the segments 
$s_{ij}^k$ and $s_{i'j'}^{k'}$ intersect the curve $p^k\cap p^{k'}$  at the same points, 
this will produce some special gluing of pairs of pants,
all other gluings (with other angles of rotation of $p^k$ with respect to $p^{k'}$) will be compared with this special gluing;
\item[3)] for arbitrary gluing the angles of rotation are compared with the chosen special gluing, 
when the angle is changed by $2\pi$ one obtains the same hyperbolic structure on the surface, but the different point of the 
Teihm\"uller space. 
 
\end{itemize}

So, the Fenchel-Nielsen coordinates on $\T$ build from the pants decomposition $P$
consist of $3g-3+2n$ length parameters $l(c_i)$ (lengths of all the curves $c_i\in P$ 
including the boundary curves of $S$)
and $3g-3+n$ angle parameters $\alpha(c_j)$ (angles along all non-boundary curves $c_j\in P$, $c_j\notin \partial S$). 
We denote $$FN(P)=\{l(c_i),\alpha(c_j) \ | \ c_i\in P;\ \ c_j\in P,  c_j\notin \partial S\}.$$
%where $\tau\in \T$ is a point at  which all  angle parameters are zero (there is a freedom in the choice of this point: by adding 
%constants to the parameters $\alpha(c_j)$ one can turn any point into the point with all zero angle parameters).   
We will also assume that the Dehn twist along $c_j$ changes $\alpha(c_j)$ by $2\pi$.

The construction establishes the homeomorphism between $\T$ and $\R_{>0}^{3g-3+2n}\times \R^{3g-3+n}$ (where $\R_{>0}$ stays
for positive real numbers).

\begin{remark}
After the Teichm\"uller space $\T$ is introduced using any given pants decomposition $P_0$ (or even using a marking of other type),
one can choose any pants decomposition $P$ to introduce the coordinates $FN(P)$ on the same space $\T$.

\end{remark}

\medskip
\noindent
Our aim is to transform Fenchel-Nielsen coordinates to coordinates containing only length parameters.

\begin{definition}[{{\it Locally parametrizing decomposition}}]
We say that a double pants decomposition $DP$ is {\it locally parametrizing} at the point $\tau\in \T$ if the functions
$l(DP)=\{l(c)\ | \ c\in DP \}$ provide a local homeomorphism from  a neighborhood of $\tau$ to a neighborhood of some point in 
$\R^{6g-6+3n}$.
By a {\it chart} $\C(DP)$ we mean a pair $(X,l(DP))$ where $X$ is the 
set of points $\tau\in \T$  such that $DP$ is locally parametrizing at $\tau$.

\end{definition}

Our first aim is  to prove that admissible double pants decompositions are locally parametrizing.
As an intermediate technical step in the proof we will use {\it mixed} coordinates, containing some angle-parameters 
(but less than Fenchel-Nielsen coordinates). 

\begin{definition}[{{\it Mixed coordinates}}]
\label{mixed}
Let $DP=(P_a,P_b)$ be a double pants decomposition, possibly with some double curves.
Let $FN(P_b)$ be some Fenchel-Nielsen coordinates build from $P_b$.
Denote by $mix(DP,FN(P_b))$ the following set of functions: 
$$mix(DP,FN(P_b))=\{l(c),\alpha(c')\ | \ c\in DP, c'\in P_a\cap P_b \},$$
where $\alpha(c')$ is the corresponding angle coordinate in $FN(P_b)$.

\end{definition}

\section{Some properties of length functions}
\label{sec tichnic}
In this section we prove several facts from hyperbolic geometry. In particular, Lemmas~\ref{handle} and~\ref{flip}  
will be crucial for the construction of locally parametrizing double pants decompositions.
Lemmas~\ref{all flips}--\ref{congruent hexagons} are preparatory.
We will denote the hyperbolic plane by $\H^2$.

\begin{lemma}
\label{all flips}
Let $S=S_{0,4}$, let $c,d\in S$  be two closed curves $|d\cap c|=2$. Let $P$ be a pants decomposition of $S$, $c\in P$.
Suppose that $d'\in S$ is a curve obtained from $c$ by a flip of $P$.
Then $d'=t_{c}^{k}(d)$ for some integer $k$, where $t_c$ is a Dehn twist along $c$.

\end{lemma}

The lemma follows immediately from~\cite[Lemma~1.16]{FN}.

\begin{lemma}
\label{glide}
Let $p\in \H^2$ be a line separating points $O$ and $O'$. Given the distances from $p$ to $O$ and $O'$, the distance $OO'$ is a monotonic
 function on the distance $PP'$, where $P$ and $P'$ are the orthogonal projections of points $O$ and $O'$ to $p$.  

\end{lemma}

\begin{proof}
Suppose that the points $P$ and $O$ are fixed, and the point $P'$ (together with $O')$ glide away from $P$, see Fig.~\ref{fig glide}.b.
Then the point $X=OO'\cap p$ glide away from $P$ which implies that the distance $OX$ grows monotonically when $PP'$ increases.
By the similar reason $O'X$ grows, and hence, $OO'$ grows monotonically.   

\end{proof}

\begin{figure}[!h]
\begin{center}
\psfrag{a}{\small (a)}
\psfrag{b}{\small (b)}
\psfrag{P}{\scriptsize $P$}
\psfrag{P'}{\scriptsize $P'$}
\psfrag{O}{\scriptsize $O$}
\psfrag{O'}{\scriptsize $O'$}
\psfrag{X}{\scriptsize $X$}
\psfrag{p}{\scriptsize $p$}
\epsfig{file=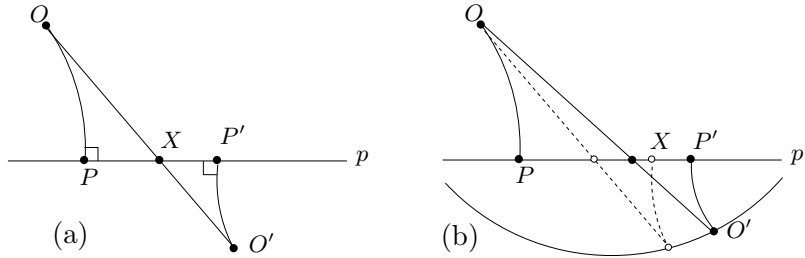,width=0.68\linewidth}
\caption{To the proof of Lemma~\ref{glide}} 
\label{fig glide}
\end{center}
\end{figure}

\begin{lemma}
\label{congruent hexagons}
Let $S=S_{0,3}$ be a three-holed sphere with a boundary $\partial S= c_1\cup c_2\cup c_3$, and let 
$s_{ij}$ be a segment orthogonal to $c_i$ and $c_j$, for $i\ne j$, $i,j\in \{1,2,3\}$. Then the segments $s_{12}, s_{13}, s_{2,3}$ 
decompose $S$ into two congruent right-angled hexagons.

\end{lemma} 

\begin{proof}
It is clear that the segments $s_{ij}$ decompose $S$ into two right-angled hexagons.  
Since a right-angled hexagon is determined (up to an isometry) by the lengths of three non-adjacent sides (the lengths of
 $s_{12}, s_{13}, s_{2,3}$), the hexagons are congruent.

\end{proof}

If the curves $a,b\in S$ are orthogonal to each other we will write ``$a\perp b$''.

\begin{lemma}
\label{handle}
Let $S=S_{1,1}$ be a handle with a boundary curve $c$, let  $a,b\subset S$ be two curves $|a\cap b|=1$.
Then the set of functions $\bar x=(l(a),l(b),l(c))$ is a local coordinate on $\T\setminus X $ where $X=\{\tau\in \T| a\perp b \}$. 
Moreover, $\bar x$  determines the point $\tau\in\T$ up to at most two possibilities.  

\end{lemma}

\begin{proof}
Shortly speaking, the coordinates  $\bar x=(l(a),l(b),l(c))$ are produced from Fenchel-Nielsen coordinates.
 More precisely, we fix Fenchel-Nielsen coordinates $FN(P)=(l(a),\alpha(a),l(c))$ arising from pants decomposition $P=\{a,c\}$.
We fix some values of  $l(a)$ and $l(c)$ and denote by $\alpha_0$ the value of $\alpha(a)$ at the point where $l(a)$ and $l(c)$
have the chosen values and $a$ is orthogonal to $b$.
We will show that  $l(b)$ 
is a monotonic function on the absolute value $|\alpha(a)-\alpha_0|$, which will imply all statements of the lemma.
Below we explain this in more details.   

First, we cut $S$ along $a$ and obtain a pair of pants $S'$ with three boundary components $c$, $a$ and $a'$.
For each of the three pairs of boundary components of $S'$
we draw a segment orthogonal to both of these two components. Denote these segments by $s_{c,a}$,  $s_{c,a'}$,  $s_{a,a'}$,
see Fig.~\ref{ruchka}.a. 
 The three segments decompose $S'$ into two right-angled hexagons $H_1$ and $H_2$. Similarly, together with the curve $a$
the three segments decompose the initial handle $S$ into two hexagons.

Consider the covering of $S$ by hyperbolic plane. We are interested in the tiling of the plane by the images of  $H_1$ and $H_2$. 
Notice that the  copies of $H_1$ and $H_2$ adjacent along the image of $s_{a,a'}$ (or $s_{c,a}$ or $s_{c,a'}$)  have this side in 
common, while the gluing along the images of $a$ and $a'$ depends on the angle parameter $\alpha(a)\in FN(P)$. 
More precisely, when $\alpha(a)= \alpha_0$  the adjacent along $a$
hexagons have a common side, otherwise the hexagons are shifted one along another  
as  in Fig.~\ref{ruchka}.b. With growth of $\alpha(a)$ the hexagons in one row glide monotonically along the hexagons of the 
other row. We denote by $p$ and $p'$ the lines separating the rows. 

Now, consider the curve $b\in S$, $|b\cap a|=1$. First, suppose that $b\perp a$, i.e. the image $\hat b$ of $b$ in the hyperbolic plane 
coincide with the image $AA'$ of $s_{a,a'}$. 
%We fix  some values $l_b$ and $l_c$ of $l(b)$ and $l(c)$ and.
%choose the value  $\alpha_0$ of $\alpha(a)$ so that $a$ is orthogonal to $b$ at the point $(\alpha_0,l_b,l_c)$. 
Now, we increase $\alpha(a)$ and look at the image $\hat b\in \H^2$ of $b$: 
since $b$ is a closed geodesic on $S$, $\hat b$ is a line forming the same angles with $p$ and $p'$.
This implies that $\hat b$ passes through the midpoint $O$ of $AA'$. Hence, $AY=A'Y'$,
where $Y=\hat b\cap p$ and $Y'=\hat b\cap p'$.
Furthermore, the hexagon $H_2'$ is shifted with respect to the hexagon $H_2$ to the distance 
$\rho=l(a) \frac{(\alpha(a)-\alpha_0)}{2\pi}$. 
Denote by $T$ the vertex of $H_2'$ projecting to the same point of $S$ as $A'$ (as  in Fig.~\ref{ruchka}.b), then
$TY=AY=A'Y'$. Hence,
$AY=1/2\rho = l(a)  \frac{(\alpha(a)-\alpha_0)}{4\pi}$.  
The same formula holds for any positive  value of $(\alpha(a)-\alpha_0)$ as well as for any negative one (in the latter case 
the point $Y\in l$ lies on the other side with respect to $A$). 

This implies that the distance $YY'=l(b)$ grows monotonically with the growth of $|\alpha(a)-\alpha_0|$:
$$
\cosh \frac{YY'}{2}=\cosh OY=\cosh OA \cosh AY= \cosh OA \cosh ( l(a) \frac{\alpha(a)-\alpha_0}{4\pi})  
$$
Hence, $|\alpha(a)-\alpha_0|$  may be recovered from $l(b)$. 
So, given the lengths $(l(a),l(b),l(c))$  one may find the Fenchel-Nielsen coordinates $FN(P)$ 
up to two possibilities. In particular, in the neighborhood of a point $\tau\in \T$ where  
$a$ is not orthogonal to $b$, the sign of $(\alpha(a)-\alpha_0)$ does not changes, which implies that the functions
 $(l(a),l(b),l(c))$ form a local coordinate  in $\T\setminus X$, $X=\{\tau\in \T| a\perp b \}$.

\end{proof}

\begin{figure}[!h]
\begin{center}
\psfrag{a_}{\small (a)}
\psfrag{b_}{\small (b)}
\psfrag{d}{\scriptsize $c$}
\psfrag{a}{\scriptsize $a$}
\psfrag{a'}{\scriptsize $a'$}
\psfrag{s_aa'}{\scriptsize $s_{aa'}$}
\psfrag{s_da}{\scriptsize $s_{ca'}$}
\psfrag{s_da'}{\scriptsize $s_{ca'}$}
\psfrag{H1}{\scriptsize $H_1$}
\psfrag{H2}{\scriptsize $H_2$}
\psfrag{H1'}{\scriptsize $H_1'$}
\psfrag{H2'}{\scriptsize $H_2'$}
\psfrag{A}{\scriptsize $A$}
\psfrag{A'}{\scriptsize $A'$}
\psfrag{O}{\scriptsize $O$}
\psfrag{Y}{\scriptsize $Y$}
\psfrag{Y'}{\scriptsize $Y'$}
\psfrag{l}{\scriptsize $p$}
\psfrag{l'}{\scriptsize $p'$}
\psfrag{T}{\scriptsize $T$}
\epsfig{file=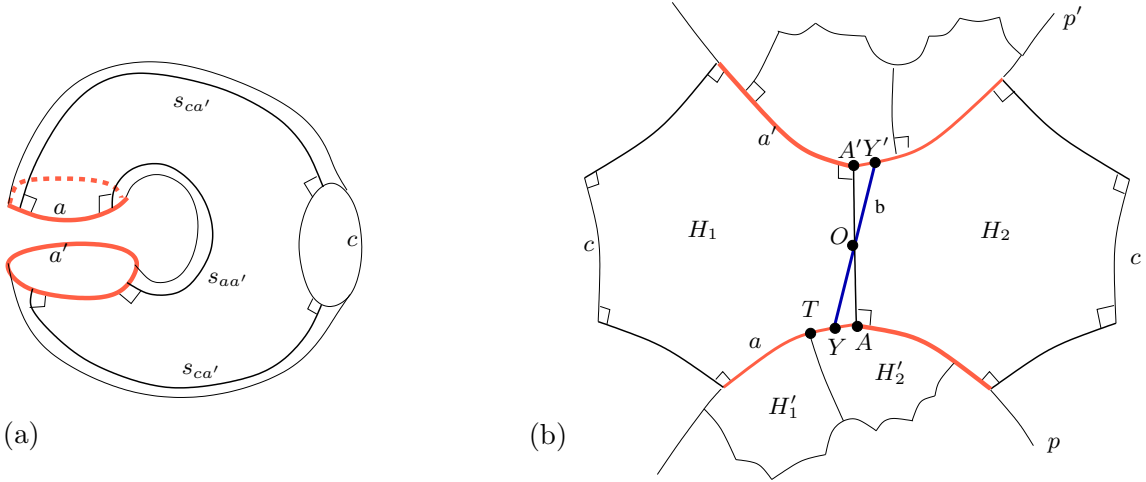,width=0.98\linewidth}
\caption{Length coordinates on a handle} 
\label{ruchka}
\end{center}
\end{figure}

\begin{remark}
\label{limit1}
Given Fenchel-Nielsen coordinates $(l(a),\alpha(a),l(c))$ on the handle, for each pair of lengths $l_0(a)$ and $l_0(c)$ 
there exists a unique angle $\alpha_0(a)$ such that $a$ is orthogonal to $b$.

\end{remark}

\begin{lemma}
\label{flip}
Let $S=S_{0,4}$ be a sphere with four holes, with boundary curves $c_1,c_2,c_3,c_4$. Let $a\in S$ be a closed geodesic
and let $b\in S$ be a closed geodesic obtained from the curve $a$ by a flip.
Then 
\begin{itemize}
\item[(1)] the angle formed by $a$ and $b$ is of the same size for  both intersections of $a$ and $b$;
\item[(2)] the set of functions $\bar x=(l(a),l(b),l(c_1),l(c_2),l(c_3),l(c_4))$ 
is a local coordinate on $\T\setminus X $ 
where $X=\{\tau\in \T | a\perp b \}$; 
\item[(3)] $\bar x$  determines the point $\tau\in\ \T$ up two at most two possibilities.  

\end{itemize}
\end{lemma}

\begin{proof}
The idea of the proof is the same as in the proof of Lemma~\ref{handle}: the coordinate $\bar x$ is obtained from Fenchel-Nielsen
coordinates $FN(P)$ built from pants decomposition $P=\{a,c_1,c_2,c_3,c_4\}$. We show that given the values of
$(l(a),l(c_1),l(c_2),l(c_3),l(c_4))$ 
the length $l(b)$ is a monotonic function on the absolute value
$|\alpha(a)-\alpha_0|$, where $\alpha_0$ is the value of $\alpha(a)\in FN(P)$ at the point of $\T$ such that 
$a$ is orthogonal to $b$ 
(and the values of $(l(a),l(c_1),l(c_2),l(c_3),l(c_4))$  are the chosen ones).
Hence, $l(b)$ determines $\alpha(a)$ up to 2 possibilities. Moreover, in the neighborhood of a point $\tau\in \T$ 
where $|\alpha(a)-\alpha_0|\ne 0$, the sign of $(\alpha(a)-\alpha_0)$ is determined uniquely by the sign at $\tau$.

In more details, the curve $a$ decompose $S$ into two pairs of pants, and each pair of pants is decomposed into two right-angled 
hexagons (respectively, by the segments $s_{ac_1}, s_{c_1c_2},s_{c_2a} $ and  $s_{a'c_3}, s_{c_3c_4},s_{c_4a'} $ orthogonal to a pair of 
boundary components), see Fig.~\ref{flip1}.a. The images of four right-angled hexagons tile the covering hyperbolic
plane: two hexagons adjacent by the image of the side $a$ are shifted  by the distance $\rho=l(a) \frac{\alpha(a)-\alpha_0}{2\pi}$
along the line containing the images of $a$, see  Fig.~\ref{flip1}.b. 

Denote by $O$ and $O'$ the midpoints of images of $s_{c_1,c_2}$ and $s_{c_3,c_4}$. Notice that the symmetry in the point $O$ 
preserves the tiling of the hyperbolic plane by hexagons (compare with Lemma~\ref{congruent hexagons}).
The same holds for the symmetry in $O'$. Consider a line $OO'$ and its intersection with the images of the curve $a$. 
It is easy to see that all angles made by $OO'$ and images of $a$ are equal. Furthermore, $OO'$ intersects the images of 
 $s_{c_1,c_2}$ and $s_{c_3,c_4}$ always in midpoints (to see that consider an image $O''$ of $O$ with respect to the symmetry in $O'$:
it lies on $OO'$ and in the midpoint of some image of $s_{c_1c_2}$, then consider the image of $O'$ with respect to a symmetry in 
$O''$ and so on). This implies that the line $OO'$ is the union of images of some closed geodesic $c\in S$, $|c\cap a|=2$. 
Hence, $c$ may be obtained from $a$ by a flip. 
Notice that $c$ intersects $a$ in two points, forming two angles of the same size. 
The length $l(c)=2\cdot OO'$ increases as $|\alpha(a)-\alpha_0|$ increases
(the distances from the points $O$ and $O'$ to the line $p$ remain constant, but one point glide along $p$ with respect to the other,
so that we may apply Lemma~\ref{glide}).

Increasing the angle $\alpha(a)$, we increase the shift between the adjacent hexagons.
Increasing $\alpha(a)$ by $2\pi$ we obtain the initial tiling of the plane by hexagons, but the line $OO'$  in the new picture is moved,
so that it is an image of another closed curve $c'\in S$ which may be obtained from $a$ by a flip.  Increasing (or decreasing)
 $\alpha(a)$ by $2\pi k$ we run through all curves on $S$ which may be obtained by a flip from $a$ 
(compare with Lemma~\ref{all flips}).
In particular, for some value of $k$ we obtain the curve $b$. This implies statement (1). 
So, the length $l(b)$ increases with growth of $|\alpha(a)-\alpha_0|$. Hence $l(b)$ determines $\alpha(a)$ up to two possibilities,
which implies that the set of functions $\bar x$ determines Fenchel-Nielsen coordinates $FN(P)$ up to two 
possibilities. This proves statement (3).
 If $b$ is not orthogonal to $a$ at $\tau\in \T$ then in the neighborhood of $\tau $ the function
$l(b)$ (together with the chosen value of $\alpha(a)$ at $\tau$) determines completely the function $\alpha(a)$, 
which implies that $\bar x$ is a set of local coordinates,
and statement (2) is also proved.

\end{proof}

\begin{figure}[!h]
\begin{center}
\psfrag{a_}{\small (a)}
\psfrag{b_}{\small (b)}
\psfrag{c1}{\scriptsize $c_1$}
\psfrag{c2}{\scriptsize $c_2$}
\psfrag{c3}{\scriptsize $c_3$}
\psfrag{c4}{\scriptsize $c_4$}
\psfrag{a}{\scriptsize $a$}
\psfrag{a'}{\scriptsize $a'$}
\psfrag{s_ac1}{\scriptsize $s_{ac1}$}
\psfrag{s_c1c2}{\scriptsize $s_{c_1c_2}$}
\psfrag{s_c2a}{\scriptsize $s_{c_2a}$}
\psfrag{s_a'c3}{\scriptsize $s_{a'c3}$}
\psfrag{s_c3c4}{\scriptsize $s_{c_3c_4}$}
\psfrag{s_c4a'}{\scriptsize $s_{c_4a'}$}
\psfrag{A}{\scriptsize $A$}
\psfrag{B}{\scriptsize $B$}
\psfrag{C}{\scriptsize $C$}
\psfrag{D}{\scriptsize $D$}
\psfrag{O}{\scriptsize $O$}
\psfrag{O'}{\scriptsize $O'$}
\psfrag{p}{\scriptsize $p$}
\epsfig{file=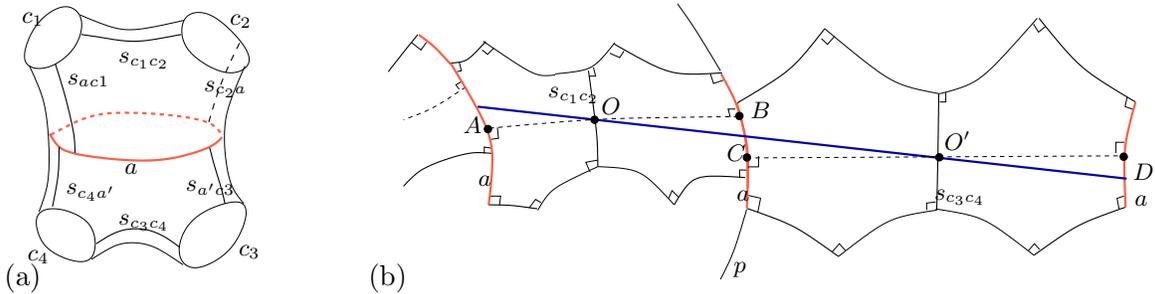,width=0.98\linewidth}
\caption{Length coordinates on a four-holed sphere} 
\label{flip1}
\end{center}
\end{figure}

\begin{remark}
\label{limit2}
Given Fenchel-Nielsen coordinates on $S_{0,4}$, 
for each lengths $l_0(a)$ together with fixed lengths of the boundary components of $S_{0,4}$ 
there exists a unique angle $\alpha_0(a)$ such that $a$ is orthogonal to $b$.

\end{remark}

\section{Locally parametrizing double pants decompositions}
\label{sec local}

In this section we prove Theorem~\ref{local} which states that for an admissible double pants decomposition $DP$ the 
functions $l(DP)$ provide  
a local parameter in neighborhoods of almost all points $\tau\in \T$.

The proof of the theorem is inductive. In Section~\ref{ex}, we build some examples of locally parametrizing 
double pants decompositions. These examples called {\it special decompositions} will be the base of the induction. 
In section~\ref{sec reduction},
we show that any admissible double pants decomposition may be obtained from a special one by a sequence of flips.
Finally,  in Section~\ref{sec flips}
we show that flips preserve the parametrizing properties of double pants decompositions.

\subsection{Examples of locally parametrizing double pants decompositions}
\label{ex}
In this section we present an example of a  locally parametrizing double pants decomposition for each surface $S_{g,n}$.
This will provide a base for the inductive proof of Theorem~\ref{local}.
The construction is obtained as a modification of Fenchel-Nielsen coordinates.

\begin{definition}[{{\it Special decomposition, conjugate curves}}]
%\label{special}
A double pants  decomposition $DP=(P_a,P_b)$ is  {\it special} with the {\it standard part} $P_b$ if
the following holds:
\begin{itemize}
\item[(1)] $DP$ contains no double curves;
\item[(2)] the part $P_b$ is standard;
\item[(3)] $DP$ may be obtained from a strictly standard double pants decomposition $DP_0$   
via a sequence of $m=3g-3+n$ flips $f_1,\dots,f_m$ of the $P_a$-part.

\end{itemize}

For a special decomposition $DP=(P_a,P_b)$ we will say that 
a curve $a_i\in P_a$ is {\it conjugate} to a curve $b_i\in P_b$ if either $a_i$ is obtained by a flip $f_i$ from $b_i$
or $a_i$ and $b_i$ belong to the same handle in the standard decomposition $P_b$. 
In the former case $(a_i,b_i)$ will called a {\it flip-conjugate pair}, 
in the latter case $(a_i,b_i)$ will called a {\it handle-conjugate pair}.

\end{definition}

See Fig.~\ref{bas} for an example of a special decomposition. 
Notice, that any special double pants decomposition is admissible.

\begin{figure}[!h]
\begin{center}
\psfrag{1}{\scriptsize $1$}
\psfrag{2}{\scriptsize $2$}
\psfrag{3}{\scriptsize $3$}
\psfrag{4}{\scriptsize $4$}
\psfrag{5}{\scriptsize $5$}
\psfrag{6}{\scriptsize $6$}
\psfrag{7}{\scriptsize $7$}
\psfrag{8}{\scriptsize $8$}
\epsfig{file=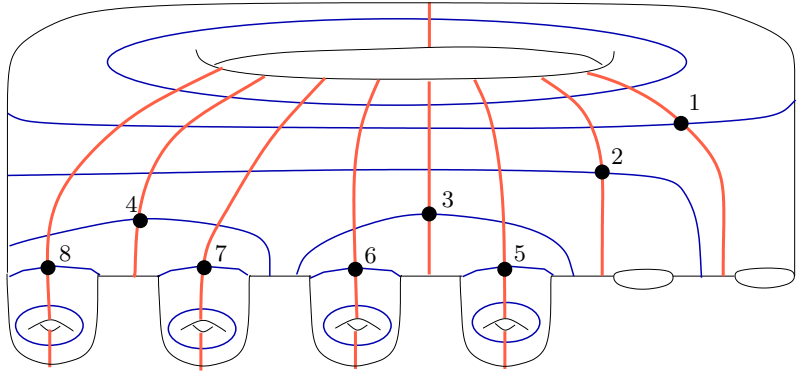,width=0.68\linewidth}
\caption{Example of a special double pants decomposition.
The black nodes show the intersections of the conjugate curves. 
The number near the nodes show the sequence of flips taking the strictly standard decomposition to the special one.} 
\label{bas}
\end{center}
\end{figure}

\begin{lemma}
\label{special}
For each standard pants decomposition $P_b$ 
 there exists a special double pants decomposition $DP=(P_a,P_b)$.

\end{lemma}

\begin{proof}
To build the required decomposition we consider a strictly standard double pants decomposition $DP'=(P_a',P_b)$ containing $P_b$ and 
apply a flip of the $P_a$-part to each of the double curves.

\end{proof}

\begin{notation}
Let   $DP=(P_a,P_b)$  be a special double pants decomposition. Denote by $Z(DP)\in \T$ the locus of points where
$a_i$ is orthogonal to $b_i$ for at least one pair of conjugate curves $(a_i,b_i)\in DP$.
\end{notation}

\begin{remark}
\label{limit}
Let $(a_i,b_i)$ be a pair of conjugate curves in a special double pants decomposition.
Remarks~\ref{limit1} and~\ref{limit2} imply that the locus of points where $a_i$ is orthogonal to $b_i$ is
homeomorphic to a hyperplane in  $\T=\R^{3g-3+2n}_{>0}\times \R^{3g-3+n}$ (here Remarks~\ref{limit1} and~\ref{limit2} work for cases
of handle-conjugate and flip-conjugate pairs respectively).
Therefore, the set
$Z(DP)\in \T$ is homeomorphic to a union of $3g-3+n$ hyperplanes
in $\T=\R^{3g-3+2n}_{>0}\times \R^{3g-3+n}$.

\end{remark}

\begin{lemma}
\label{base-param}
Let   $DP=(P_a,P_b)$  be a special double pants decomposition. Then 
\begin{itemize}
\item[(1)]  $l(DP)$ is a local coordinate in $\T\setminus Z(DP)$;
\item[(2)]  $l(DP)$ determine the point in $\T$ up to at most $2^{3g-3+n}$ choices.

\end{itemize}
\end{lemma}

\begin{proof}
Suppose that $P_b$ is a standard part of $DP$.
Choose Fenchel-Nielsen coordinates $FN(P_b)$ based on the pants decomposition $P_b$. It is a global coordinate on $\T$.
We will substitute angle coordinates of $FN(P_b)$ by length coordinates one by one.

Let $f_1,\dots,f_{m}$ be the sequence of flips described in the Definition~\ref{special}, let $b_1,\dots,b_{m}$ be 
the curves of $P_b$ such that $f_i$ is a flip applied to $b_i$. 
Let $DP_i=f_i\circ\dots\circ f_1(DP_0)$, where $DP_0$ is the corresponding strictly standard double pants decomposition.
Applying  Lemma~\ref{handle} sufficiently to all handle-conjugate pairs of curves $a_i,b_i\in DP$ we see that
$mix(DP_0,FN(P_b))$ is a local coordinate away from $Z(DP_0)$ and defines the coordinate $FN(P_b)$  up to $2^g$ choices.
Then, applying Lemma~\ref{flip} to each pair of flip-conjugate curves successively (more precisely, to the subsurface $S_{0,4}$ obtained 
by a union of two pairs of pants adjacent to $b_i$ in $P_a$-part of $DP_i$), we see 
that $mix(DP_i,FN(P_b))$ 
is a local coordinate away from $Z(DP_i)$ and defines $mix(DP_{i-1},FN(P_b))$ up to 2 choices.
This implies the lemma.

\end{proof}

\subsection{Induction step: reduction to flips}
\label{sec reduction}

\begin{lemma}
\label{only flips}
Let $DP$ be an admissible double pants decomposition. Then there exists a sequence of flips   $f_1,\dots,f_k$
such that  $DP_0=f_k\circ\dots\circ f_1(DP)$ is a strictly standard double pants decomposition.

\end{lemma}

\begin{proof}
Since $DP$ is an admissible decomposition, there exists a sequence of flips taking $DP$ to a standard 
double pants decomposition. It is known that flips act transitively on pants decompositions of $S_{0,k}$ (see~\cite{HT}),
which implies that any strictly standard double pants decomposition may be transformed to a strictly standard ones by flips.

\end{proof}

\begin{lemma}
\label{without double}
Let $DP$ be a double pants decomposition containing no double curves.
Suppose that $DP'=f_k\circ \dots \circ f_1 (DP)$, where $f_i$, $i=1,\dots,k$, is a flip. 
If $DP'$ contains no double curves
then there exists a sequence of flips $g_1,\dots,g_{r}$ such that   $DP'=g_r\circ \dots \circ g_1 (DP)$
and  no of the decompositions
$g_i\circ\dots\circ g_1(DP)$, $i=1,\dots,r$ contains double curves. 

\end{lemma}

\begin{proof}
Denote  $DP=(P_a,P_b)$ and  $DP'=(P_a',P_b')$
We will use the fact that flips of the $P_a$-part commute with flips of the $P_b$-part.

Let $C=\{c \ | \ c\in DP_i=f_i\circ \dots \circ f_1 (DP), 0\le i\le k \}$ be a set of all curves appearing during the transformation
from $DP$ to  $DP'=f_k\circ \dots \circ f_1 (DP)$. 

First, for each of the curves $a_i\in P_a$ we apply a flip $g_i$ 
so that $g_i(a_i)\notin C$: this is possible,  since $C$ is a finite set, while a set of flips for a given curve $a_i$ in a given pants 
decomposition is either infinite or empty (in the later case,  $a_i$ lies in a handle bounded by some other curve $a_j$, so we can first
destroy the handle applying a flip to $a_j$, and then apply a flip to $a_i$). Denote by $P_a''$ 
the obtained $P_a$-part of the decomposition.  

Second, we  transform $P_b$ to $P_b'$ by the same sequence of flips as in $f_1,\dots,f_k$.

Third, there exists a sequence $f_1',\dots,f_l'$ of flips taking $P_a''$ to $P_a'$.
 Denote $\xi=f_l'\circ \dots \circ f_1'$. 
Denote $C'=\{c \ | \ c\in DP''=f_i'\circ \dots \circ f_1' (DP), 0\le i\le l \}$.
For each of the curves $b_i\in P_b'$ we apply a flip $g_i'$ 
so that $g_i'(b_i)\notin C'$.

Next,  we  transform $P_b$ to $P_b'$ by the same sequence of flips as in $f_1,\dots,f_k$.

Finally, we apply the inverse sequence  $\xi^{-1}$ to take  the $P_b$-part back to the state $P_b'$.

Clearly, we can not obtain double curves at any stage of the transformation, so the lemma is proved.
 
\end{proof}

Lemma~\ref{only flips} together with Lemma~\ref{without double} imply the following lemma.

\begin{lemma}
\label{only flips without double}
Let $DP$ be an admissible double pants decomposition without double curves. 
Then there exists a special double pants decomposition $DP'$ and a
sequence of flips $f_1,\dots,f_k$ such that $DP_0=f_k\circ\dots\circ f_1(DP)$ and no of the decompositions
$f_i\circ\dots\circ f_1(DP)$, $i=1,\dots,k$, contains double curves.

\end{lemma}

\subsection{Induction step: flips}
\label{sec flips}  
In this section we show that flips take locally parametrizing double pants decompositions to locally parametrizing ones.

In the next lemma we show this property for almost all flips.

\begin{lemma}
\label{almost all}
Let $DP$ be a parametrizing double pants decomposition at $\tau\in \T$. Let  $f'$ and $f''$ be two different flips of the same curve $c\in DP$, 
such that neither $DP'=f'(DP)$ nor $DP''=f''(DP)$ contain double curves. 
If $DP'$ is not parametrizing at $\tau\in\T$  then $DP''$ is parametrizing at $\tau$.

\end{lemma}

\begin{proof}
Let $DP=(P_a,P_b)$, $c\in P_a$. Let   $DP'=(P_a',P_b)$,  $DP''=(P_a'',P_b)$. Denote by $c'$ and $c''$ the curves of $P_a'$ and 
$P_a''$  obtained from $c$ by flips $f'$ and $f''$ respectively. 
In addition, denote by $S_*$ a subsurface of $S$ composed of two pairs of pants in $P_a$ adjacent to the curve $c$. 

Suppose that $DP'$ is not a parametrizing double pants decomposition at $\tau\in \T$. By definition, this means that 
there exists a 
non-trivial deformation $\xi(\tau)$ of the hyperbolic structure,  where $\xi$ preserves all lengths of curves contained in $(P_a',P_b)$.
This deformation may be described as a set of simultaneous  small twists along the curves of $P_a'$ 
(the rates of the twists need not coincide or to be constant). 

Suppose that $\xi$ contains no twist along $c'$ (i.e. the twist along this curve is trivial, zero). Then
the subsurface $S_*$ is not changed, and the length of the curve $c$ is preserved by $\xi$. Hence, $\xi$ preserves the 
lengths of all curves in $(P_a,P_b)=DP$. By assumption, these lengths provide a local coordinate at $\tau$, 
so the deformation $\xi$ is trivial (does not change the point of Teichm\"uller space).
The contradiction shows that $\xi$ contains a non-trivial twist along $c'$.

On the other hand, consider another deformation $\eta$ of the initial hyperbolic structure $\tau\in \T$, 
where $\eta$ preserves all lengths of curves from $(P_a,P_b)$ except the length of $c$. A locus of points of $\T$ obtained by $\eta$ from $\tau$
is a 1-dimensional curve in a neighborhood of $\tau$. This implies that $\eta=\xi$. 

Suppose now that $DP''$ also is not parametrizing at $\tau$. Similarly to the case of $DP'$,
this implies that there exists a deformation $\psi$ preserving all lengths of curves from $P_a''$ and containing
a non-trivial twist along the curve $c''\in P_a$. Similarly to $\xi$, the deformation $\psi$ should coincide with $\eta$, 
so, $\xi=\psi$. However, these two transformations do not coincide in the subsurface $S_*$: one twists along $c'$,
another along $c''\ne c'$.
The contradiction shows that the double pants decomposition $DP''$ is parametrizing at $\tau$.
  
\end{proof}

\begin{lemma}
\label{step-flip}
Let $DP$ be a locally parametrizing double pants decomposition at $\tau\in \T$. Let $f_0$ be a flip of $DP$
such that the double pants decomposition $DP^{(0)}=f_0(DP)$ contains no double curves. 
Then $DP^{(0)}$ is a parametrizing double pants decomposition  at $\tau\in \T$.
 
\end{lemma}

\begin{proof}

Let $DP=(P_a,P_b)$ be a parametrizing double pants decomposition  at $\tau\in \T$. 
Let $c\in DP$ be a curve flipped by $f_0$.
Without loss of generality we may assume that $c\in P_a$. Denote $m=3g-3+n$.

Consider an $m$-dimensional surface $C_a$ through $\tau\in \T$ such that the lengths of all curves contained in $P_a\setminus P_b$ 
are constant in $C_a$. Let $C_b$ be a similar surface for $P_b$. Denote by $\Pi_a$ and $\Pi_b$ the tangent planes to $C_a$ and 
$C_b$ in $\tau$. Let $C_{\partial S}$ be an $n$-dimensional surface through $\tau$ such that all curves contained in $\partial S$ have
constant lengths in $C_{\partial S}$, let $\Pi_{\partial S}$ be the corresponding tangent plane.
Since $DP=(P_a,P_b)$ is  parametrizing at $\tau$,  the planes $ \Pi_a$, $\Pi_b$ and $\Pi_{\partial S}$ intersect each other in $\tau$
only (and span the whole tangent space at $\tau$).

Let $\psi_i$, $i=1,\dots,m$, be the curves in $\T$ on which  all lengths of curves of $DP$ are preserved
except for the length of one curve $b_i\in P_b\setminus P_a$. 
Let $\bar b_1,\dots,\bar b_m$ be the tangent vectors to $\psi_1,\dots,\psi_m$ at $\tau$.
Clearly, the plane $\Pi_a$ is spanned by the vectors  $\bar b_1,\dots,\bar b_m$.

Now, consider a series of flips $f_i$ of the curve $c\in P_a$ (including the flip $f_0$ described in the lemma): 
we will assume that the flip $f_i$ takes $c$ to the curves $c_i$ of the same homology class; moreover, we assume that 
$c_{i+1}$ may be obtained from $c_i$ by a Dehn twist along $c$. For each of the flips $f_i$ we denote $P_a^{(i)}=f_i(P_a)$.
Denote by $\Pi_a^{(i)}$ the tangent planes at $\tau$ to the surfaces of the constant lengths of curves from $P_a^{(i)}\setminus P_b$. 

If the double pants decomposition $DP^{(0)}=(P_a^{(0)},P_b)$ is parametrizing at $\tau$, then there is nothing to prove.
So, suppose that $DP^{(0)}$ is not parametrizing at $\tau$. By Lemma~\ref{almost all}, 
this implies that all other double pants decompositions
$DP^{(i)}=(P_a^{(i)},P_b)$ are parametrizing at $\tau$ (with possible exclusion of at most one decomposition $DP^{(j)}$: 
at most one of these decompositions may contain a double curve $c_i$).
Reasoning as above with $\Pi_a$, we show that the plane $\Pi_a^{(i)}$  is spanned by  $b_1,\dots,b_m$.
This implies that for $i\notin \{0,j\}$ all planes  $\Pi_a^{(i)}$ coincide with  $\Pi_a$.

Now, our aim is to show that  $\Pi_a^{(0)}=\Pi_a$.
Let $t_c$ be a Dehn twist along $c$. The twist $t_c$ takes $c_i$ to $c_{i+1}$. On the other hand, $t_c$ acts on $\T$
and takes $\Pi_a^i$ to $\Pi_a^{i+1}$. Since $\Pi_a^i=\Pi_a$ for $i\notin \{0,j\}$, $t_c$ preserves $\Pi_a$.
Hence, $\Pi_a^i=\Pi_a$ for all $i\in \Z$.     

Since  $\Pi_a^{(0)}=\Pi_a$, the planes $\Pi_a^{(0)}$, $\Pi_b$ and $\Pi_{\partial S}$ span the tangent space at $\tau$, 
which implies that  $DP^{(0)}=(P_a^{(0)},P_b)$ is a parametrizing double pants decomposition at $\tau$.

\end{proof}

\begin{theorem}
\label{local}
Let $DP$ be an admissible double pants decomposition without double curves.
Then $DP$ together with the ordered set of lengths $l(DP)=\{l(c_i)| c_i\in DP \}$ is a local coordinate in $\T\setminus Z(DP')$
for some special double pants decomposition $DP'$.

\end{theorem}

\begin{proof}
By Lemma~\ref{only flips without double} there  exists a special double pants decomposition $DP'=(P_a',P_b')$,  and a sequence $\psi$
of flips taking $DP$ to $DP'$ and producing no double curves on its way. By Lemma~\ref{base-param} the lengths $l(DP')$ form a local 
coordinates in $\T\setminus Z(DP')$. By Lemma~\ref{step-flip} each of the flips
in the sequence $\psi$ preserve the parametrizing property of double pants decomposition (i.e. the obtained decomposition provides
 a local parameter in  $\T\setminus Z(DP')$). Hence, $DP$ is parametrizing in  $\T\setminus Z(DP')$.

\end{proof}

\section{Finite number of choices}
\label{sec finite}
In Section~\ref{sec local}, we proved that the set of functions $l(DP)$ is a local parameter in almost all points of $\T$.
In this section, we prove that the functions  $l(DP)$  determine the point of $\T$ up to finitely many choices.

Consider an universal covering $\pi$ of $S$ by a hyperbolic plane $\H^2$, so that $S=\H^2/G$ where $G\in SL_2(\R)$ 
is some finitely generated discrete group.
Let $M_1,\dots, M_s$ be a finite set of  matrices generating $G$. Let  $r_1(M_1,\dots,M_s)=\dots =r_n(M_1,\dots,M_s)=E$ be the
defining relations, where $r_i$ is a word in the alphabet $A=\{M_1,M_1^{-1},\dots,M_s,M_s^{-1}\}$. 

For each closed geodesic $c\subset S$ each connected component of the preimage $\pi^{-1}c$ is a line (denote it by 
$L_i(c)$, where integer index stays to emphasize that there are countably many of these preimages).
The group $G$ contains a hyperbolic transformation $\gamma(c)$ shifting $\H^2$ along $L_i(c)$  for the distance equal to $l(c)$.
So, we have 
\begin{equation}
\label{eq}
\tr (\gamma(c))=2 \cosh(l(c)/2). 
\end{equation}

Notice that $\gamma(c)=w(M_1,\dots,M_s)$ for some word $w$  in the same alphabet $A$.
So, the Formula~\ref{eq} may be considered as a finite set of polynomials in matrix elements of $M_1,\dots,M_s$ with coefficients
$$\hat l(c)=2 \cosh(l(c)/2).$$ 

\begin{theorem}
\label{finite}
Let $DP$ be an admissible double pants decomposition containing no double curves. 
Then $l(DP)$ determines a point of $\T$ up to finitely many choices.

\end{theorem}

\begin{proof}
For each of the curves $c_i\in DP$ we consider one of its preimages on $\H^2$ together with the hyperbolic transformation 
$\gamma(c_i)$. 
Taking in account Formula~\ref{eq}, we obtain a system of polynomial equations in elements of $M_i$: the system consists of the equations
arising from the following three sources:

\begin{itemize}
\item[(1)] $M_i\in SL_2(\R)$;
\item[(2)] $r_j(M_1,\dots,M_s)=E$, where $r_k$ is one of the defining relations;
\item[(3)] $\tr (\gamma(c))=\hat l(c)$. 
\end{itemize}

The matrix equations of the second type are considered as four scalar equations in matrix elements.
Notice, that the equations of all three types are polynomial (here we use the fact that  $M_i\in SL_2(\R)$, and hence, all elements
of $M_i^{-1}$ are also elements of $M_i$).
So, the three types compose a system of finitely many polynomial equations in matrix elements of $M_i$ with integer coefficients 
and constant terms in $\Z\cup \{\hat l(c_1)),\dots,\hat l(c_m)\}$.  Suppose in addition that 
the values  of  $(\hat l(c_1)),\dots,\hat l(c_m))$ correspond to at least one  hyperbolic structure $\tau\in \T$ on $S$. 
Then the system of equations is solvable. On  the other hand, Theorem~\ref{local} implies that the system is non-generate.
Thus, there are finitely many solutions of this system. 

In other words, for each set of values $l(DP)$  we can write a unique set of values   
$\hat l(DP)=\{ \hat l(c_1)),\dots,\hat l(c_m)\}$; for this set $\hat l(DP)$ there are
finitely many possible values of matrix elements of $M_i$. So, for each value of $l(DP)$ 
there are finitely many distinct points in $\T$. 

\end{proof}

\begin{cor}
\label{alg}
Let $DP$ be an admissible double pants decomposition of $S$ without double curves. 
Let $c\in S$ be a closed curve. Then $\hat l(c)$ is an algebraic function of $\hat l(DP)$.  

\end{cor}

\begin{proof}
By Theorem~\ref{finite} the value of $l(DP)$ determines the point of $\T$ up to finitely many choices. 
Each of these choices correspond to a unique (modulo conjugation) discrete subgroup  $G\in SL_2(\R)$ acting on $\H^2$.
Consider the group $G$ for one of these possibilities.  

Following the proof of Theorem~\ref{finite} consider a preimage $L(c)$ of $c$ in $\H^2$ and a hyperbolic transformation $\gamma(c)$
which shifts along $L(c)$ by the distance $l(c)$. Then $\gamma(c)=w(M_1,\dots,M_s)$ where $w$ is a word in the alphabet 
$\{M_i,M_i^{-1}| i=1,\dots,s\}$. So, $\hat l(c)=\tr \gamma(c)$ is a polynomial in the matrix elements of $M_1,\dots,M_S$.
Since the elements of matrices $M_i$ are the solution of a system of polynomial equations, these elements are algebraic functions
of $\hat l(DP)$. This implies, that $\hat l(c)$ is an algebraic function of $\hat l(DP)$ either.     

\end{proof}

\section{An atlas on the Teichm\"uller space}
\label{atlas}

In Section~\ref{sec flips}, we proved that for each admissible double pants decomposition $DP$ the function $l(DP)$ provides 
a local coordinate in neighborhoods of almost all points in $\T$ (more precisely, away from a set of measure 0 formed by a finite union 
of hypersurfaces). 
In this section, we show that the coordinate charts with coordinates $l(DP)$ compose an atlas on $\T$. Moreover, the transition functions
between the adjacent chart change exactly one coordinate (and correspond to flips and handle twists of double pants decompositions).

\begin{lemma}
\label{twist helps}
Let $S$ be a surface with a fixed hyperbolic structure. Let  $DP=(P_a,P_b)$ be a special double pants decomposition
 with a standard part $P_b$.
Let $a_i,b_i\in DP$ be a pair of conjugate curves in $DP$. Let $b_j\in P_b$ be a curve such that $b_j\cap a_i\ne \emptyset$ 
and 
let $t_{b_j}$ be a Dehn twist along $b_j$. 
If $a_i$ is orthogonal to $b_i$ then $t_{b_j}^{k}(a_i)$ is not orthogonal to $b_i$ for all $k\in \Z\setminus 0$.

\end{lemma}

\begin{proof}
First, notice that if $i=j$ than there is nothing to prove (the statement follows than from Lemma~\ref{handle} in case of 
handle-conjugate curves and from Lemma~\ref{flip} in case of flip-conjugate curves).
From now on we assume $i\ne j$.

Notice that by construction of special decompositions, the condition  $b_j\cap a_i\ne 0$ implies that $(a_i,b_i)$ can not be a pair of
handle-conjugate curves.  So,
 $(a_i,b_i)$ is a  pair of flip-conjugate curves, and the curve $b_i$ is homologically trivial. 
Suppose $b_i$ is orthogonal to $a_i$ as well as to $t_{b_j}^{k}(a_i)$, where $k\ne 0$.
Since $b_i$ is homologically trivial,  $b_i$ cuts $S$ into two connected components $S_1$ and $S_2$.
Let $S_1$ be the component containing the curve $b_j$. Denote  $s=a_i\cap S_2$ and  $s'=t_{b_j}^{k}(a_i)\cap S_2$.  
In view of Lemma~\ref{flip} all ends of $s$ and $s'$ are orthogonal to $b_i$. 

Since $b_j\in S_1$, and $b_j\cap b_i=\emptyset$, the topology of the decomposition of $S_2$ is not changed by $t_{b_j}$ 
(however, geometrically $s\ne s'$). This implies that 
there exists an isotopy $\gamma_x$ of $s$ to $s'$ (where $x\in [0,1]$, $\gamma_0=s, \gamma_1=s'$)
such that the ends of the segment $\gamma_x(s)$ belong to $b_i$. 
Notice that $s$ can not intersect $s'$, otherwise the segments $s,s'$ and a part of $b _j$ bound a hyperbolic triangle with two right 
angles $b_js$ and $b_js'$, which is impossible. On the other hand, if $s\cap s'=\emptyset$ then two parts of $b_j$, $s$ and $s'$ bound 
a hyperbolic quadrilateral with four right angles, which is also impossible. The contradiction shows the lemma.

\end{proof}

\begin{lemma}
\label{cover}
For each point $\tau\in \T$ there exists a double pants decomposition $DP_{\tau}$ such that $l(DP_{\tau})$ is a local coordinate in a 
neighborhood of $\tau$.

\end{lemma}

\begin{proof}
Consider an arbitrary special double pants decomposition $DP=(P_a,P_b)$ with a standard part $P_b$. 
By Theorem~\ref{local}, $l(DP)$ is a local coordinate in $\T\setminus Z(DP)$.  
So, if  $\tau\notin Z(DP)$ then there is nothing to prove. 
Suppose that  $\tau\in Z(DP)$, i.e. there exists an orthogonal  conjugate pair of curves $a_i,b_i\in DP$,
(a pair of conjugate curves such that $a_i$ is orthogonal to $b_i$ in $\tau$). 
We will apply to $DP$ a twist $t_{b_j}$ in some of the curves $b_j\in DP$ in order to reduce the number of 
orthogonal conjugate pairs. 

To see that it is always possible,  suppose that  $a_i,b_i\in DP$ is an orthogonal  conjugate pair. 
In this case there exists an integer $k$ such that the special decomposition $t_{b_i}^k(DP)$ 
contains less orthogonal conjugate pairs than $DP$ has (the pair $a_i,b_i$ of this twisted decomposition is not orthogonal for each 
$k\ne 0$, Lemma~\ref{twist helps} implies that for all but finitely many values of $k$ the $k$-th degree of the twist will not produce 
new orthogonalities for other conjugate pairs). 

\end{proof}

Lemma~\ref{cover} shows that the charts with coordinates $l(DP)$ cover the space $\T$. 
Now, we consider the transition functions between the charts. In view of Theorem~\ref{trans}, 
it is natural to choose these transition functions as ones induced by flips and handle-twists of admissible double pants decompositions. 

The case of flip is considered in Lemma~\ref{step-flip}: it is shown that as long as a flip $f$ produces no double curves, $f$ preserves
the locus of points where the set of functions $l(DP)$ is a local coordinate. We have also shown in Lemma~\ref{only flips without double} 
that if $DP$ and $DP'$ are two double pants decompositions containing no double curves and $DP$ can be turned into $DP'$ by a 
sequence of flips, than one can choose this sequence of flips so that no double curves are produced on the way.

It is impossible to treat handle-twists directly in the same way:
by definition no  handle-twist can be applied to a double pants decomposition containing no double curves. To overcome this obstacle,
we introduce the notion of a {\it quasi-handle-twist}.

\begin{definition}[{{\it Quasi-handle-twist}}]
\label{quasi}
Let $DP$ be a double pants decomposition without double curves. Let $c\in DP$ be a curve such that there exists a flip $f(c)$ 
producing a handle $\h$ in the decomposition $f(DP)$ (so that $f(c)$ is a double curve which cuts out the handle).
Let $a\in DP\cap f(DP)$  be a curve contained in the handle $\h$. By a {\it quasi-handle-twist} $t_a$ of $DP$ we 
mean a Dehn twist along $a$.

\end{definition}

\begin{remark}
\label{rem twist}
The quasi-handle-twist $t_a$ may be written as $t_a=f^{-1}\circ \hat t_a \circ f$, where $f$ is a flip as in Definition~\ref{quasi} and
$\hat t_a$ is a handle twist in the handle $\h$.

\end{remark}

\begin{remark}
\label{rem twist of T}
%A Dehn twist changes a double pants decomposition on the surface with fixed hyperbolic structure.
%On the other hand, the same twist may be considered as a transformation on $\T$ preserving the pants decomposition $DP$.
%(the curve on $\T$ connecting the initial decomposition with the endpoint is easy to construct in Fenchel-Nielsen coordinates
%built from the part $P_a$ of $DP$, where $a\in P_a$: to follow te curve, one increase monotonically the angle coordinate in 
%$\alpha (a)$ from $0$ to $2\pi$).
%
Since $t_a$ is a Dehn twist, $t_a$ acts on the Teichm\"uller space $\T$. 
We denote by $t_a(\tau)$ the point of $\T$ obtained from $\tau$ by the Dehn twist $t_a$.

\end{remark}

Now, we will prove the counterparts to the Lemmas~\ref{step-flip} and~\ref{only flips without double} 
for the case of quasi-handle-twists.

The next Lemma follows immediately from Definition~\ref{quasi} and Remarks~\ref{rem twist} and~\ref{rem twist of T}.

\begin{lemma}
\label{quasi-twist}
Let $DP$ be an admissible double pants decomposition without double curves. Let $\tau\in\T$ be a point such that $l(DP)$ is a local 
coordinate in $\tau$. Let $t$ be a quasi-handle-twist along the curve $c\in DP$. Then  
$l(t(DP))$ is a local coordinate in $\tau'=t(\tau)$.

\end{lemma}

\begin{lemma}
\label{chart-trans}
Let $DP$ and $DP'$ be two admissible double pants decompositions containing no double curves.
Then there exists a sequence of flips and quasi-handle-twists which takes $DP$ to $DP'$ and produces no double curves on its way. 

\end{lemma}

\begin{proof}
By Theorem~\ref{trans} there exists a sequence $\psi$ of flips and handle-twists taking $DP$ to $DP'$. In view of 
Lemma~\ref{without double}, each subsequence containing no handle-twist may be realized without producing double curves.
It is sufficient to prove the lemma for the case when $\psi$ contains one handle-twist only (and then apply inductional reasoning).
Suppose that this unique handle-twist $\hat t_c$ is a twist in a curve $c\in DP^{st}$ where $DP^{st}$ is a standard double pants 
decomposition flip-equivalent to $DP$ (it is shown in~\cite[Lemma 4.1]{FN} that  handle-twists in standard decompositions are 
sufficient for obtaining the transitivity theorem). 

Let $DP^{st}=(P_{a}^{st},P_{b}^{st})$ be the two parts, suppose that $c\in P_{a}^{st}$. Let $DP^{sp}=(P_{a}^{sp},P_{b}^{sp})$ 
be a special decomposition with the standard part $P_{b}^{sp}=P_{b}^{st}$. By Lemma~\ref{without double}, there
 exists a sequence of
flips taking $DP$ to $DP^{sp}$ without producing  double curves. Then we apply a quasi-handle-twist $t_c$ in $c$, so that we obtain   
another special decomposition $ DP^{sp}_*$. In view of Remark~\ref{rem twist}, $DP^{sp}_*$ is flip-equivalent to $DP'$. 
The sequence of flips taking  $DP^{sp}_*$ to $DP'$ without producing double curves does exist in view of 
Lemma~\ref{without double}.

\end{proof}

%Lemma~\ref{cover} imply that the charts defined by the functions $l(DP)$ (for different admissible double pants decompositions
%containing no double curves)
%compose an atlas on $\T$. In view of Lemma~\ref{chart-trans} the flips and  quasi-handle-twists of double pants decompositions
%act transitively on these charts (so that the transition functions are algebraic, see Corollary~\ref{alg}).
%On the other hand, each of the local charts $l(DP)$ (for each individual admissible double pants decomposition without double curves) 
%determines a local chart in almost all points of $\T$, which motivates the following definition:

%\begin{definition}[{{\it Almost global chart}}]
%By an {\it almost global chart} on $\T$ we mean a set of functions $\bar f=(f_1,\dots,f_s)$ satisfying the following properties:
%\begin{itemize}
%\item[1)]  $\bar f$ is defined in each point of $\T$ and continuous on $\T$;
%\item[2)]  there exists a codimension 1 subset  $X\subset \T$ such that
%$\T\setminus X$ consists of finitely many connected parts and
% $\bar f$ is a global coordinate in each of the parts.
%
%\end{itemize}

%\end{definition}

%Given an admissible double pants decomposition $DP$ without double curves,
%by a {\it chart} $\C(DP)$ we call a union of points $\tau\in \T$ where the functions $l(DP)$ provide a local coordinate in
%some neighborhood of $\tau$.  

Summarizing results of Lemmas~\ref{cover},~\ref{chart-trans} and Corollary~\ref{alg} we obtain the following theorem.

\begin{theorem}
\label{trans on charts}
(1) The charts $\C(DP)$ with coordinates $l(DP)$, where $DP$ is an admissible double pants decomposition without double curves,
provide an atlas on Teichm\"uller space $\T$.

(2) The elementary transition functions of these charts are induced by  flips and quasi-handle-twists of double pants decompositions,
each elementary transition function changes only one coordinate. 
This unique non-trivial transition function is algebraic.

(3) The compositions of elementary transition functions act transitively on the charts.

\end{theorem}

\section{Deligne-Mumford compactification of moduli space}
\label{compactification}

In Section~\ref{atlas}, we showed that the Teichm\"uller space is covered by coordinate charts arising from admissible double pants 
decompositions. Since local coordinates on Teichm\"uller space are also local coordinates on the moduli space, the charts
with coordinate $l(DP)$  also compose an atlas on the moduli space. 
In this section, we show that this atlas works also for most strata in Deligne-Mumford  compactification of the moduli space.

Consider some Fenchel-Nielsen coordinates  $FN(P)$ on the Teichm\"uller space 
$$\T=\{l(c_i)>0, \ \alpha(c_j)\in \R \ |\ c_i,c_j\in P, c_j\notin \partial S \}.$$
Given a pants decomposition $P$ denote
$$\T_P=\{l(c_i)\ge 0, \ \alpha(c_j)\in \R \ |\ c_i,c_j\in P, c_j\notin \partial S \}.$$
The {\it  augmented Teichm\"uller space} $\overline \T$ is the following closure of $\T$: 
$$\overline \T=\cup_{P}\T_P $$
where the union is taken by all pants decompositions of the surface.
The points of $\overline \T\setminus \T$ correspond to {\it nodal surfaces}, i.e. to the surfaces with {\it nodal singularities}:
a nodal singularity arises when a non-trivial closed curve $c$ in $S$ is degenerated to a point (i.e. $l(c)\to 0$). 
A nodal surface is not a surface: a neighborhood of a nodal point is not homeomorphic to a disk.
We denote by $N$ the set of all nodal points on the nodal surface.
It is known that  $\overline \T /Mod= \overline \M$, where $Mod$ is the modular group and $\overline \M$ is the Deligne-Mumford 
compactification of the modular space $\M=\T/Mod$.

The space $\overline \T$ inherits topology from $\cup_P T_p = \cup_P (\R_{\ge 0}^{3g-3+2n}\times \R^{3g-3+n})$.

Given an admissible double pants decomposition $DP$ without double curves, we say that the {\it boundary of the chart } $\C(DP)$ 
is the locus of points $\tau'\in  \overline \T$  where  $l(c)=0$ for at least one $c\in DP$.

\begin{theorem}
\label{closure}
For each point $\tau' \in \overline \T$ there exists an admissible double pants decomposition $DP$ containing no double curves
and such that  $\tau'$ belongs to the boundary of the chart $\C(DP)$ with coordinates $l(DP)$. 

\end{theorem}

\begin{proof}
If $\tau' \in\T$, then there is nothing to prove in view of Theorem~\ref{trans on charts}. Suppose that 
$\tau' \in (\overline \T \setminus \T)$. Then there exists a set of mutually disjoint curves $C$  on $S$
such that the surface $S'$ corresponding to $\tau'$ is obtained by contracting all curves $c_i\in C$.
It is sufficient to show that there exists an admissible double pants decomposition $DP$ containing no double curves
and such that $C\in DP$.

Consider any pants decomposition $P_a$ containing the set $C$. We will build the required decomposition
$DP=(P_a,P_b)$ in the following four steps: 
first, we transform $P_a$ by a sequence of flips to a standard 
decomposition $P_a'$; second, we build a standard double pants decomposition $(P_a',P_b')$; next, we transform $P_a'$ back to
$P_a$ by flips; finally, we apply (if necessary) several flips to $P_b'$ to avoid double curves.  
 
\end{proof}

Factorizing by the modular group $Mod$ we obtain the charts on the Deligne-Mumford compactification of the modular
space (with the natural notion of the {\it boundary of the chart} on $\overline \M$ defined as the boundary of the same chart on 
$\overline \T$ factorized by $Mod$).
Applying the same reasoning as in Theorem~\ref{closure} we obtain the following corollary.

\begin{cor}
\label{cor}
For each point $\tau' \in \overline \M$ there exists an admissible double pants decomposition $DP$ containing no double curves
and such that  $\tau'$ belongs to the boundary of the chart $\C(DP)$ with coordinates $l(DP)$. 

\end{cor}

\begin{remark}
For many of the points $\tau'\in \partial \overline \T$ the coordinates $l(DP)$ provide also a chart in a neighborhood 
$O'(\tau')=O(\tau')\cap \partial \overline \T$  (where $O(\tau')$ is some neighborhood of $\tau'$ in $\overline \T$.
It would be natural to try to cover $\partial \overline \T$ (resp. the whole boundary of $\overline \M$)
by these charts. However, in general it turns to be impossible
(see Remark~\ref{ex-impossible}). 

Below, we define a large subset of ``good'' points in the boundary and 
show that all points of this subset are covered by the charts $\C(DP)$. 

\end{remark}

\medskip

The boundary $\partial \overline \T$ is stratified: given a set $C$ of mutually non-intersecting curves in $S$, 
a {\it stratum} $\S_C$ is a locus $\{l(c_i)=0 \ | \ c_i\in C\}$. 
All nodal surfaces (of genus $g$ with $n$ boundary parts) 
with $k$ nodal singularities compose a union $\S_{2k}$ of codimension $2k$ strata. 
%We will be most interested in maximal dimension strata $\S_2$.

By a pants decomposition of a surface $S$ with punctures we mean a decomposition into generalized pairs of pants
(where a generalized pair of pants is either a sphere with three holes, or a sphere with two holes and a puncture, or a sphere
with a hole and two punctures, or a sphere with three punctures).

By a {\it pants decomposition} (respectively, {\it double pants decomposition})  of a nodal surface $S$ we mean a set of curves $P$ 
%including all nodal points $c'\in N$ (as curves of zero length) and 
composing a pants decomposition 
(respectively double pants decomposition) in all {\it components} of $S$ (connected components of $S\setminus N$).
The nodal points are not considered as curves of the pants decomposition.
%In the case of double pants decomposition $DP$, each nodal point belong to each part of $DP$.

A (double) pants decomposition of $S$ is {\it standard} if the decompositions of all components are standard. 
Similarly, a double pants decomposition is {\it special} if decompositions of all components are special.

Let $DP=(P_a,P_b)$ be a double pants decomposition of $S$ containing no double curves.
Let $c\in DP$. Denote by $S'$ the nodal surface obtained
from $S$ by collapsing $c$ to a nodal singularity. Consider a set $DP'$ of curves on $S'$ obtained as a union of images of curves of $DP$
 which do not intersect $c$. Notice that
 $DP'$ is not necessarily a double pants decomposition of $S'$; however, if $c\in P_b$ intersects only one other curve of $DP$
then $DP'$ is. 

%\begin{lemma}
%\label{conj}
% If $c\in P_b$ intersects only one other curve $c'\in P_a$ and $c'$ may be obtained from $c$ by a flip of $P_a$ or by a composition of
%handle-twists of $DP$ then  $DP'$ is a double pants decomposition of $S'$.
%Moreover, if $DP$ is special, then $DP'$ is special either.
%?????????????????????
%\end{lemma}
%
%The proof is evident: applying a flip (or a composition of handle-twists) to $P_a$ we obtin a pants decomposition containing $c$,
%so, after the shrinking $c$ we obtain a pants decomposition $P_a'$ of $S'$.
%
%Furthermore, suppose that $C=\{c_1,\dots,c_k \}\subset P_b$ is a subset of $P_b$ such that for each $i\in \{1,\dots,k \}$ the curve 
%$c_i$ intersects only one curve $c_i'$ of $P_a$ (and such that the pair $(c_i,c'_i)$ satisfy the conditions of Lemma~\ref{conj}). 
%Then shrinking of all curves $c_i\in C$ leads to a nodal surface $S_C$ with a double pants 
%decomposition $DP_C$ obtained by shrinking $DP$.

\begin{lemma}[(Collar Lemma, \cite{Ke})]
Let $c \in S$ be a simple closed geodesic on hyperbolic surface $S$ of lengths $l=l(c)$. Define $w$ by the relation 
$$\sinh l \sinh w=1.$$
Then $S$ contains a collar $Col(c)$ of width $w$ defined by  $Col(c)=\{x\in S \ | \ \rho_S(x,c)<w/2 \}$,
where $\rho_S(A,B)$ is the distance in $S$ from the set $A$ to the set $B$.

\end{lemma}

It follows immediately from the Collar Lemma that if 
$a,b\in S$ are closed geodesics $b\cap a\ne \emptyset$
then contracting $a$ so that $l(a)\to 0$ implies $l(b)\to \infty$.

The Collar Lemma implies that the local coordinates $l(DP)$ degenerate while the curves $c_i$ are collapsing: 
if $a_i\cap c_i\ne \emptyset$  then $l(a_i)\to \infty$ while $c_i\to 0$  (the curve $a_i$ intersecting $c_i$ do exists since 
$c_i\notin P_a$ and $P_a$ is a maximal set of disjoint curves in $S$).
For the case $C\subset DP$ 
we define the new set of functions $\tilde l(DP,C)$  as follows:
$$ \tilde l(DP,C)=\{l(c_i), \frac{1}{l(c_j)} \ | \ c_i\in C, c_j\in  DP\setminus C  \}.$$
Clearly, $\tilde l(DP,C)$ is a local coordinate in all points of $\T$, where $l(DP)$ is a local coordinate. 
Moreover, this set of functions remains correctly defined while the curves of the set $C$ are collapsed.

\begin{definition}[{{\it Inversion}}]
An {\it inversion} of a $k$-th function of  $\tilde l(DP,C)$ is an exchange of $l(c_k)$ or  $\frac{1}{l(c_k)}$ (where $c_k\in DP$)
by  $\frac{1}{l(c_k)}$ or $l(c_k)$ respectively. 

\end{definition}

It is clear the transformation from  a set of functions $\tilde l(DP,C)$ to any other set of functions  $\tilde l(DP',C')$
may be obtained as a composition of inversions and transformations induced by flips and quasi-handle-twists 
of double pants decompositions
(here $DP$ and $DP'$ are admissible double pants decompositions containing no double curves, $C$ and $C'$ are sets of disjoint curves).

\begin{definition}[{{\it Strong and weak curves}}]
Let   $C=\{c_1,\dots,c_k\}$ be a set of mutually disjoint curves on $S$.
Each curve $c\in C$ appears two times in the boundary of $S\setminus C$. We say that $c$ is a {\it strong curve} of $C$ 
if two copies of $c$ 
appear in two different connected components of $S\setminus C$. Otherwise, we say that $c$ is {\it weak}. 

We denote by $C_{strong}\subset C$ the subset of all strong curves. 

\end{definition}

We denote by $S^1,\dots,S^l$ the connected components of $S\setminus C$.
By $\hat S^i$ we denote the connected component of $S\setminus C_{strong}$ corresponding to the component $S^i$ of $S\setminus C$
($\hat S^i$ is obtained from $S^i$ by gluing along the pairs of boundary components arising from the weak curves).

\begin{definition}[{{\it Good set of curves}}]
\label{good}
We say that a set  $C=\{c_1,\dots,c_k\}$ of mutually disjoint curves on $S$ is {\it good} if  
each connected component $\hat S^i$ of 
$S\setminus C_{strong}$ is either a surface of positive genus or has at least two boundary components contained in $\partial S$.
%
%an $r$-holed sphere $S_{0,r}$ with more than $r-2$ curves $c_i\in C$ lying in the boundary $\partial S^i$.

Let $\S_{good}$ be a union of all strata $\S_C$ where $C$ is a good set.
\end{definition}

\begin{remark}
It is easy to see that $\S_{2}\subset \S_{good}$, where $\S_2$ is a union of all codimension 2 strata.

\end{remark}

\begin{lemma}
\label{strat}
Let $C=\{c_1,\dots,c_k\}$ be a good set of curves.
Then there exists a special double pants decomposition $DP$ of $S$ 
such that $C\subset DP$ and each curve $c_i\in C$ is intersected by a unique curve of $DP\setminus c$. 
Moreover, collapsing any curve $c_i\in C$ to a nodal singularity leads to a special double pants decomposition 
of the obtained nodal surface.
\end{lemma}

\begin{proof}
We build a special double pants decomposition $DP=(P_a,P_b)$ with a standard part $P_b$ such that $P_b$ contains all strong curves
of $C$ and $P_a$ contains all weak curves of $C$.
We construct the decomposition $DP$ separately for each connected component $\hat S^i$ of $S\setminus C_{strong}$. 

If $\hat S^i$ is a sphere with holes, then we build the decomposition $DP$ as shown in Fig.~\ref{sphere}:
since $C$ is a good set of curves, at least two boundary components of $\hat S^i$ do not belong to $C$ 
(the two bottom boundary components in the figure). In  Fig.~\ref{sphere}.a we show the part $P_a$ of $DP$,
in  Fig.~\ref{sphere}.b we show the whole decomposition $DP=(P_a,P_b)$, 
notice that each curve of $C$ is intersected by a unique curve of $P_b$.

\begin{figure}[!h]
\begin{center}
\psfrag{1}{\scriptsize $1$}
\psfrag{2}{\scriptsize $2$}
\psfrag{3}{\scriptsize $3$}
\psfrag{4}{\scriptsize $4$}
\psfrag{5}{\scriptsize $5$}
\psfrag{6}{\scriptsize $6$}
\psfrag{7}{\scriptsize $7$}
\psfrag{8}{\scriptsize $8$}
\psfrag{9}{\scriptsize $9$}
\psfrag{10}{\scriptsize $10$}
\psfrag{11}{\scriptsize $11$}
\psfrag{a}{\small (a)}
\psfrag{b}{\small (b)}
\epsfig{file=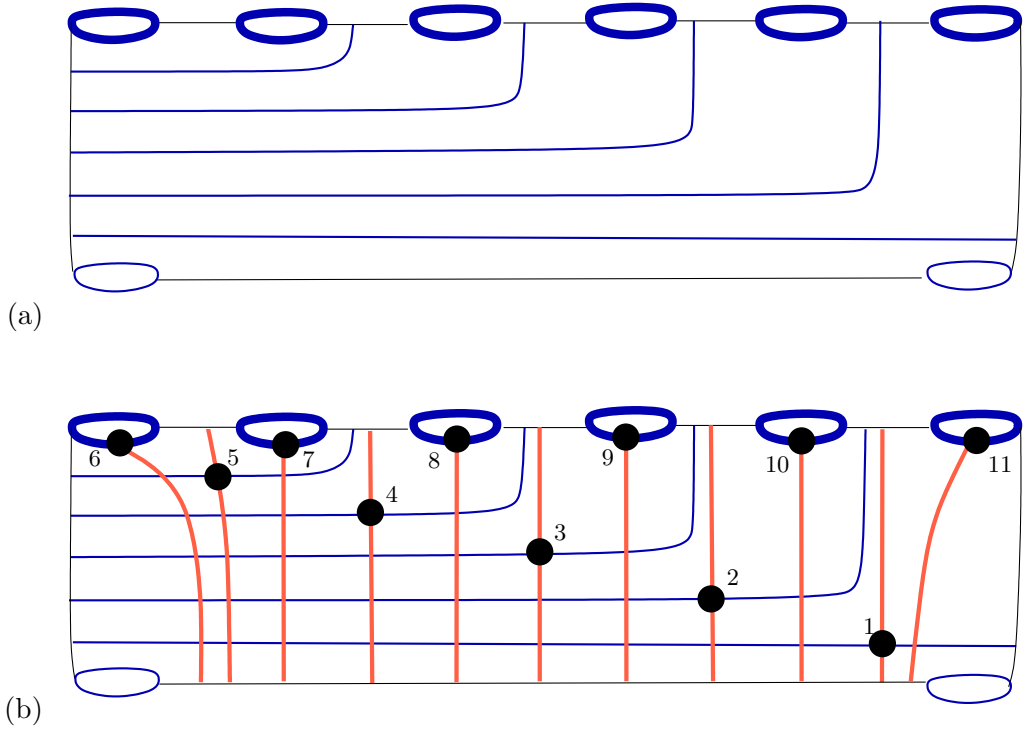,width=0.88\linewidth}
\caption{Special double pants decomposition containing $C$: case $\hat S^i=S_{0,r}$. 
The curves of $C$ are bold, each intersects a unique other curve of $DP$. 
The figure shows only the front part of the surface, the decomposition of the back part is the same.
The black nodes show the intersections of the conjugate curves.} 
\label{sphere}
\end{center}
\end{figure}

Now, suppose that $S^i$ contains at least one handle.

First, we build a standard pants decomposition $P$ containing the set $C$.  
To do this for the component $\hat S^i$, we   
build a standard decomposition  with a linear structure as in Fig.~\ref{spec}.a: 
first come all handles than come all holes.
Moreover, for each strong curve $c_j\in \hat S^i$ the curve $c_j$ is contained inside one of the handles
(more precisely, first we build the curves $\tilde c_j\in S^i$ 
which together with both copies of $c_j$ bounds a pair of pants in $S^i$,
then in $\hat S_i$ the curve $\tilde c_j$ cuts out a handle $\h_j$ containing $c_j$).

Next, we build the standard part $P_b$ of the special decomposition  $DP=(P_a,P_b)$: we take the standard decomposition $P$
and for each handle $\h_j$ of $P$ we substitute the curve $c_j\in P\cap C$ by any other curve $c_j'\in\h_j$ such that 
$|c_j\cap c_j'|=1$ (in the handles containing no curves of $C$ we do nothing).

Now, we build the part $P_a$ of the special decomposition $DP=(P_a,P_b)$. We build the restriction of $P_a$ to $\hat S^i$
as it is shown in Fig.~\ref{spec}.b: namely, each of the weak curves $c_j\in C$ is intersected only by a unique curve of $P_a$ 
lying in the same handle as $c_j$; each of the strong curves is intersected only by a curve passing through the handle $\h_0^i$.

The obtained decomposition $DP$ is special: it may be transformed to a standard decomposition by a sequence of flips
as shown in Fig.~\ref{sphere} and Fig.~\ref{spec} (we show the order of flips by numbering the intersection points of conjugate 
curves). It is easy to see that collapsing any curve $c_i\in C$ to a point we get a special double pants decompositions $DP'$ 
of the obtained nodal surface: the sequence of flips taking $DP'$ to a standard decomposition almost coincide with the corresponding 
sequence for $DP$ (the only difference is that in case of strong curve $c_i$ one needs to omit the flip in the curve conjugated to $c_i$).

\end{proof}

\begin{figure}[!h]
\begin{center}
\psfrag{1}{\scriptsize $1$}
\psfrag{2}{\scriptsize $2$}
\psfrag{3}{\scriptsize $3$}
\psfrag{4}{\scriptsize $4$}
\psfrag{5}{\scriptsize $5$}
\psfrag{6}{\scriptsize $6$}
\psfrag{7}{\scriptsize $7$}
\psfrag{8}{\scriptsize $8$}
\psfrag{a}{\small (a)}
\psfrag{b}{\small (b)}
\epsfig{file=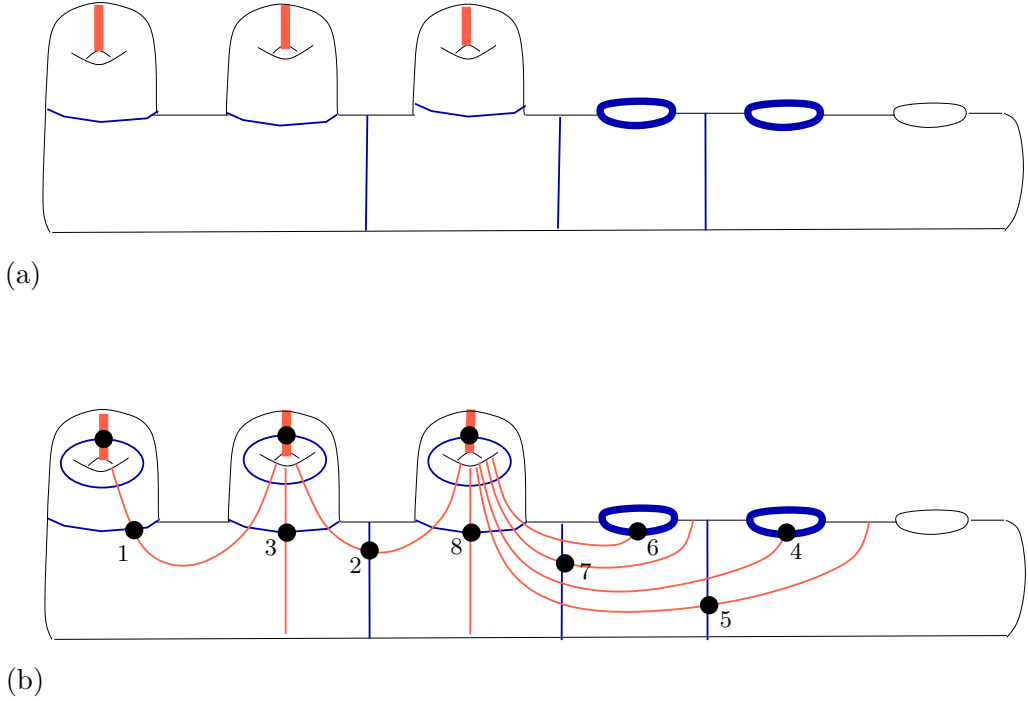,width=0.88\linewidth}
\caption{Special double pants decomposition containing $C$ 
(the curves of $C$ are bold, each intersects a unique other curve of $DP$). 
The figure shows only the front part of the surface, the decomposition of the back part is the same.
The black nodes show the intersections of the conjugate curves.} 
\label{spec}
\end{center}
\end{figure}

Let $DP$ be a special double pants decomposition of $S$, let $C\in DP$ be a good set of curves.
Denote by $\overline Z(DP,C)$ the locus in $\overline \T$ where at least one of the conjugate pairs of $DP\setminus C$ 
is an orthogonal pair.
%(in other words, $\overline Z(DP)$ is a closure of $Z(DP)$ in $\overline \T$).

\begin{remark}
Let $(a_i,b_i)$ be a conjugate pair of $DP$ and let $b_i\in C$.
It is easy to see that while $b_i$ is collapsed, the angle formed up by $a_i$ and $b_i$ tends to the right angle
(if lengths of other curves of $P_b$ remain fixed). This implies that $\S_C$ belongs to the closure of $Z(DP)$ in $\overline \T$.
Therefore, we can not hope that the set of functions $\tilde l(DP)$ will provide a local coordinate in the whole neighborhood of
a given point $\tau'\in \S_C$. 

Instead, we will show that for any point $\tau'\in \S_C$ there exists a suitable special double pants decomposition $DP$
such that $\tilde l(DP)$ is a local coordinate in the neighborhood of $\tau'$ in $\S_C$  as well as a local coordinate in almost all
points of the neighborhood of $\tau'$ in $\overline \T$ (more precisely,  $\tilde l(DP)$ is a local coordinate in 
 $O(\tau')\setminus Z(DP)$ where $O(\tau')$ is a neighborhood of $\tau'$ in $\overline T$.
%The set   $O(\tau')\setminus Z(DP)$ is homeomorphic to a half-ball with several half-hyperplanes excluded.

\end{remark}

This motivates the following definition:

\begin{definition}[{{\it Almost chart}}]
\label{almost chart}
Let $C$ be a good set of curves, let $\S_C\subset \overline \T$ be the corresponding stratum and let  
$\tau'\in \S_C$ be a point. An {\it almost chart} centered at $\tau'$ is a pair
$(O(\tau'),f)$ where $O(\tau')\subset \overline \T$ is a neighborhood  of $\tau'$ and $f=(f_1,\dots,f_k)$ 
is a set of $k$ functions,
$k=\dim \T=6g-6+3n$ satisfying the following conditions:
\begin{itemize}
\item[1)] the functions $f$ are defined and continuous in $O(\tau')$; 

\item[2)] $f$ is a local coordinate in a neighborhood $O'(\tau')=O(\tau')\cap S_C$;

\item[3)] there exists a finite set $X$ of codimension 1 surfaces in $\overline \T$ such that
$f$ is a local coordinate in a neighborhood of each point $\tau\in O(\tau')\cap (\overline \T \setminus X)$.

\end{itemize}

\end{definition}

\begin{lemma}
\label{pinching}
Let $S$ be a marked hyperbolic surface considered as a point of $\T=\T(S)$.
Let $S_{good}\subset \overline \T$ be a union of the good strata.
Let $S'$ be a nodal surface with nodal singularities, such that the marked hyperbolic structure $\tau'$ of $S'$
belongs to $\S_{good}$.

Then there exists an admissible double pants decomposition $DP$ of $S$ which degenerates to an admissible double pants decomposition
$DP'$ of $S'$ such that $\tilde l(DP',C)$ provides an almost chart centered in $\tau'$.

\end{lemma}

\begin{proof}
Since $S'$ belongs to $S_{good}$, the nodal surface $S'$ is obtained from $S$ by collapsing the curves contained in some good set $C$.

By Lemma~\ref{strat} there exists a special double pants decomposition $DP=(P_a,P_b)$
with standard part $P_b$, such that 
\begin{itemize}
\item[(1)] $c_i\in P_b$ for all strong curves $c_i$ of $C$; 
\item[(2)] $c_i\in P_a$ for all weak curves $c_i$ of $C$; 
\item[(3)] for each $c_i\in C$ the decomposition $(P_a,P_b)$ contains a unique curve $d_i$ intersecting $c_i$.

\end{itemize}

By Lemma~\ref{strat} by collapsing a curve $c_i\in C$ one obtains a special double pants decomposition of the obtained nodal 
surface, and, after collapsing all curves $c_i\in C$, we obtain a special double pants decomposition $DP'$ of $S'$.
Clearly, the set of functions  $\tilde l(DP,C)$ is defined and continuous in a neighborhood $O'$ of $\tau'$.

%f $\tau' \notin \overline Z(DP)$ then $\tau \notin \overline Z(DP)$ for each point $\tau$ in a neighborhood of $\tau'$,
%hus,  by  Lemma~\ref{base-param}
%l(DP)$ provides a local coordinate in the neighborhood of $\tau'$, and everything is proved.
%uppose that $\tau' \in \overline Z(DP)$.
Using Lemma~\ref{twist helps} (as in the proof of Lemma~\ref{cover}) we may apply to $DP$ several twists 
(along the curves of $P_b$)
so that the resulting special decomposition $DP_*=t_{c_{m}}^{k_{m}}\circ \dots \circ t_{c_1}^{k_1}(DP)$ satisfies
 $\tau'\notin \overline Z(DP'_*)$. 

Suppose that some of the twists $t_{c_j}$ changes a curve  $c\in C$. Then $c\in P_a$, so $c$ is a weak curve of $C$. 
The curve $c_j$ then is the curve conjugated to $c$ in $DP$.    
Clearly, we may substitute a degree of the twist $t_{c_j}$ by a degree of the twist $t_c$ 
so that in the resulting double pants decomposition the images of curves $c$ and $c_j$ are not orthogonal to each other.
So, after several substitutions we transform $DP_*$ to a special decomposition  $DP_{**}$
such that  $\tau'\notin \overline Z(DP'_{**})$ and $C\in DP_{**}$. This implies the conditions 2) and 3) 
of Definition~\ref{almost chart}. The condition 1) of the same definition holds for  $\tilde l(DP_{**},C)$ 
in some neighborhood $O'(\tau')\subset \overline \T$  trivially.
Hence, the pair $(O'(\tau'), \tilde l(DP_{**},C))$ provides an almost chart centered at $\tau'$.

\end{proof}

Now, consider the moduli space $\M=\T/Mod$.
A local chart in a neighborhood of $\tau\in \T$ projects to a local chart in a neighborhood of $\pi(\tau)\in \M$
(where $\pi$ is a factorization by $Mod$) unless $\tau$ is a hyperbolic structure with non-trivial automorphism group,
or, equivalently, unless $\pi(\tau)$ is an orbifold point of $\M$.
Composing this with Lemma~\ref{pinching} and Theorem~\ref{trans on charts} we obtain the following theorem:

\begin{theorem}
\label{atlas on M}
Let $S$ be a nodal surface, let $\M(S)$ be its moduli space and let $\overline \M(S)$ be the Deligne-Mumford compactification of $\M$.
Let $\S_{good}^{\M}=\S_{good}/Mod$ be the union of good strata in $\M$.    
Let $O$ be a locus of orbifold points of $\M$, let $\overline O$ be the closure of $O$ in $\overline \M$.
Then 

\begin{itemize}
\item[(1)] the charts with coordinates $\tilde l(DP,C)$  provide an atlas on $\M\setminus O$ and on $\S_{good}^\M\setminus \overline O$, 
(here  $C$ is a good set and $DP$ is an admissible double pants decomposition without double curves);

\item[(2)] each point $\tau'\in S_{good}^\M\setminus \overline O$ is covered by some almost chart  $(O'(\tau'),\tilde l(DP,C))$;

\item[(3)] the elementary transition functions of these charts (almost charts) are inversions and transformations 
induced by  flips and quasi-handle-twists of double pants decompositions;
each elementary transition function change only one coordinate; 
this unique non-trivial transition function is algebraic;

\item[(4)] the compositions of elementary transition functions act transitively on the union of charts and almost charts.

\end{itemize}

\end{theorem}

\begin{remark}
\label{ex-impossible}
We do not claim that the Definition~\ref{good} of the good strata exhaust all the points of $\partial \overline \T$ 
(resp. $\partial \overline \M$) covered by the almost charts of our atlas. 
However, some restrictions for the ``good'' points covered by the atlas  
are indispensable. For example, if $S=S_{3,0}$ and $C$ is a set of three curves cutting a pair of pants out of $C$
(see Fig.~\ref{ex-imp}) then it is possible to prove that in each 
admissible double pants decomposition $DP$ such that $C\in DP$ 
the set of curves $\{ c_i\in DP\setminus C, \  c_i\cap C\ne \emptyset\}$ contains more than three curves.
Hence, after retracting the curves of $C$, any decomposition $DP$ contains less curves (of finite non-zero length)
than required.
This implies that the points $\tau'\in \S_C$ can not be covered by any chart of our atlas.

\end{remark}

\begin{figure}[!h]
\begin{center}
\psfrag{1}{\scriptsize $c_1$}
\psfrag{2}{\scriptsize $c_2$}
\psfrag{3}{\scriptsize $c_3$}
\epsfig{file=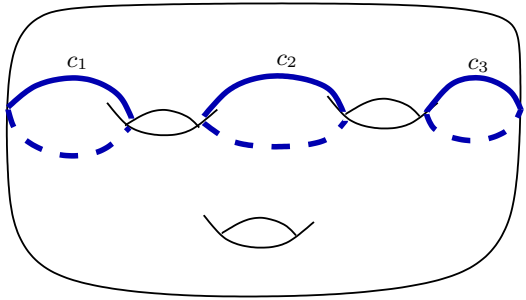,width=0.45\linewidth}
\caption{Example of the stratum not covered by the atlas: $S=S_{3,0}$, $C=\{c_1,c_2,c_3\}$.} 
\label{ex-imp}
\end{center}
\end{figure}

\end{document}